\documentclass[preprint,12pt]{article}
\usepackage[title]{appendix}
\usepackage{mathrsfs}
\usepackage{mathtools}
\usepackage{graphicx}
\usepackage{epstopdf}
\usepackage{stmaryrd}
 \usepackage{subfigure}
\usepackage{float}
    \usepackage{bm}
\usepackage{amsfonts,amssymb}
\usepackage{amsthm}

\usepackage{amsmath,amssymb,amsthm}
\usepackage{cvpr}
\usepackage{times}
\usepackage{epsfig}

\usepackage{color}
\usepackage{booktabs,threeparttable}
\usepackage{colortbl}
\usepackage[export]{adjustbox}
\usepackage{caption}

\numberwithin{equation}{section}
\newtheorem{thm}{Theorem}[section]
\newtheorem{lem}{Lemma}[section]
\newtheorem{corollary}{Corollary}[section]
\newtheorem{rmk}{Remark}[section]

\graphicspath{{FIGS/}}
\newcommand{\fr}[2]{{\frac{#1}{#2}} }
\allowdisplaybreaks[4]

\usepackage[pagebackref=true,breaklinks=true,letterpaper=true,colorlinks,bookmarks=false]{hyperref}

\cvprfinalcopy 

\def\cvprPaperID{} 

\begin{document}

\title{An extended mixed finite element method for elliptic interface problems}

\author{Pei Cao$^1$, Jinru Chen$^{1,2,*}$, Feng Wang$^{1}$}
\renewcommand{\thefootnote}{\arabic{footnote}} \footnotetext[1]{Jiangsu Key Lab for NSLSCS, School of Mathematical Sciences, Nanjing Normal University, Nanjing 210023, China}\footnotetext[2]{School of Mathematics and Information Technology, Jiangsu Second Normal University, Nanjing 211200, China}


\maketitle

\begin{abstract}
	In this paper, we propose an extended mixed finite element method for elliptic interface problems.
\textcolor[rgb]{0.00,0.00,0.00}{By adding some stabilization terms, we present a mixed approximation form based on Brezzi-Douglas-Marini element space and the piecewise constant function space, and show that the discrete inf-sup constant is independent of how the interface intersects the triangulation.
Furthermore, we derive that the optimal convergence holds independent of the location of the interface relative to the mesh.}
  Finally, some numerical examples are presented to verify our theoretical results.
\end{abstract}\\
\textbf{keyword}:
elliptic interface problems, extended mixed finite element, inf-sup condition


\section{Introduction}
Interface problems have many applications in material sciences, solid mechanics and fluid dynamics, such as heat conduction problems with different conduction coefficients, elasticity problems describing various material behaviors and two-phase flows involving different viscosities, etc. These problems usually lead to differential equations with discontinuous or non-smooth solutions across interfaces.

Due to the discontinuity of coefficients across the interface, if the standard finite element method is used to solve interface problems, one usually enforces  mesh lines along the interface in order to get optimal a priori error estimates (\cite{divex}). This is the so-called interface-fitted methods, for example, the arbitrary Lagrangian-Eulerian method in \cite{ALE}. However, for many problems in which the interface varies with time, the generation of interface-fitted meshes is costly. Therefore, it is attractive to develop numerical methods based on interface-unfitted meshes.

There have been some interface-unfitted methods, for instance, the immersed boundary method, the fictitious domain method, the immersed finite element method, the extended finite element method (XFEM), and so on. The immersed boundary method \cite{peskin2002} is robust and easy to implement, but only has lowest accuracy. The fictitious domain method (\cite{FDM95}, \cite{FDM94}, \cite{FDM09}, \cite{FDM07}, \cite{fdm19}) has several kinds of ways to impose the boundary conditions, but it is not so easy to treat complex boundary conditions (\cite{FDM09}).
As for the immersed finite element method, it is difficult to apply this method to the mixed forms of interface problems,
 for more details, please see \cite{lzl98}-\cite{Lintao19}.

Recently,
XFEM has become a popular approach to interface problems. It was originally introduced by T. Belytschko and T. Black in \cite{txlf} to solve elastic crack problems. In {\cite{02Hansbo}}, A. Hansbo and P. Hansbo firstly present the extended finite element method (Nitsche's-XFEM) based on Nitsche's method \cite{71nicai}. The main idea of Nitsche's-XFEM is to use a variant of Nitsche's method to enforce continuity on the interface.
Compared with the other methods for interface problems, Nitsche's-XFEM is usually simple and has optimal convergence.
 Whereafter, it has been successfully implemented in the field of solid mechanics and flow problems, please refer to \cite{fracture}, \cite{09els}, \cite{viewxfem1}, \cite{broadxfem}, \cite{04els}, \cite{17frac}.

Mixed finite element methods have many advantages \cite{Arnold90},
especially when one wants to have a more accurate approximation of the derivatives of the displacement.
There are many works which used the mixed finite element method to deal with Stokes interface problems, such as \cite{saddle15}, \cite{stokes2014}, \cite{Siam16}, \cite{wn2019}, \cite{wql2015}, and so on. Whereas, to the best of our knowledge, there are a few papers using the mixed finite element method to handle elliptic interface problems.
Several works related to mixed finite element methods for the elliptic interface problems are based on interface-fitted meshes, please see \cite{2020DHMFM}, \cite{wsq2007}, \cite{2000Multigrid}, \cite{Yang2003}, \cite{2019jz}. There are also some works (\cite{fracture}, \cite{17frac}) concerned with mixed finite element methods for the fracture Darcy flow based on interface-unfitted meshes.

In this paper, we propose an extended mixed finite element method for elliptic interface problems.
We use the lowest order Brezzi-Douglas-Marini finite element pair ($\mathbb{BDM}_1, \mathbb{P}_0$) to construct our extended mixed finite element space. By adding some stabilization terms, we give the extended mixed finite element approximation form, and prove that the discrete inf-sup constant is independent of how the interface intersects the triangulation. After establishing some $\bm{H}^1(\mathrm{div})$ vector extension results, we deduce the optimal convergence for our method which is independent of the location of the interface relative to the mesh.

The outline of this paper is as follows. In Section 2, we introduce the extended mixed finite element method for the elliptic interface problem. The inf-sup condition is shown in Section 3. In Section 4, the approximation properties and the optimal convergence are presented. Some numerical examples are provided in Section 5. Finally, we give the conclusions in Section 6.
\section{Preliminaries}
 \subsection{The elliptic interface problem}
We adopt the convention that the boldface represents vector-valued functions, operators and their associated spaces. Let $\Omega$ be a bounded convex polygonal domain in $\mathbb{R}^2$ with boundary $\partial\Omega$. A smooth interface defined by $\Gamma=\partial\Omega_1\cap\partial\Omega_2$ divides $\Omega$ into two open sets $\Omega_1$ and $\Omega_2$ such that $\overline{\Omega}=\overline{\Omega}_1\cup\overline{\Omega}_2$ and $\Omega_1\cap\Omega_2=\emptyset$.
Let $\alpha$ be piecewise constant defined as
\begin{equation*}
\alpha=\left\{
\begin{array}{c}
\alpha_{1}~~\mathrm{in}~\Omega_1,\\
\alpha_{2}~~ \mathrm{in}~\Omega_2.
\end{array}
\right.
\end{equation*}
Given  $f \in L^{2}(\Omega)$, we consider the following elliptic interface problem:
\begin{eqnarray}\label{exactproblem}
\begin{aligned}
 &&-\mathrm{div}(\alpha\nabla u)&=f               ~~~~~~~~~ \mathrm{in}~\Omega_1\cup\Omega_2,\\
  &&u&=0   ~~~~~~~~~ \mathrm{on}~\partial  \Omega,\\
 &&[u]&=0 ~~~~~~~~~\mathrm{on}~\Gamma,\\
 &&[\alpha\nabla u\cdot\bm{n}]&=0  ~~~~~~~~~ \mathrm{on}~\Gamma,
\end{aligned}
\end{eqnarray}
where
 $[v]=v_1|_{\Gamma}-v_2|_{\Gamma}$ is the jump on the interface $\Gamma$ with $v_{i}=v|_{\Omega_i}, ~i=1,2$, and $\bm{n}$ is the unit normal vector on $\Gamma$ pointing from $\Omega_1$ towards $\Omega_2$ (see Figure 1).

Setting $\bm{p}=\alpha\nabla u~$, then we get the mixed form:
\begin{eqnarray}\label{orig pro}
\begin{aligned}
&& -\mathrm{div}~\bm{p}&=f               ~~~~~~~~~ \mathrm{in}~\Omega_1\cup\Omega_2,\\
&& \bm{p} &=\alpha\nabla u~~~~~\mathrm{in}~\Omega_1\cup\Omega_2,\\
&&  u&=0   ~~~~~~~~~ \mathrm{on}~\partial  \Omega,\\
&& [u]&=0 ~~~~~~~~~\mathrm{on}~\Gamma,\\
&&  [\bm{p}\cdot\bm{n}]&=0  ~~~~~~~~~ \mathrm{on}~\Gamma.
\end{aligned}
\end{eqnarray}
The weak formulation of the problem (\eqref{orig pro}) is to find $({\bm{p}},u)\in \bm{Q}\times V=\bm{H}(\mathrm{div}; \Omega)\times L^2(\Omega)$ such that
 \begin{equation}\label{mixorig}
\left\{
\begin{array}{rl}
\displaystyle\int_{\Omega}\alpha^{-1}\bm{p}\cdot\bm{q}dx+\displaystyle\int_{\Omega}u\mathrm{div}\bm{q}dx&=\displaystyle 0, ~~~\quad\quad\quad\forall \bm{q}\in\bm{H}(\mathrm{div}; \Omega),\vspace{3mm} \\
 -\displaystyle\int_{\Omega}v\mathrm{div}\bm{p}dx&=\displaystyle\int_{\Omega}fvdx,  ~~~~ \forall v \in L^2(\Omega).
\end{array}
\right.
\end{equation}
It is well-known that the weak formulation (\ref{mixorig}) is well-posed (\cite{1991mixfem}).
\begin{figure*}[htbp]
  \center{
  \includegraphics[width=8cm]{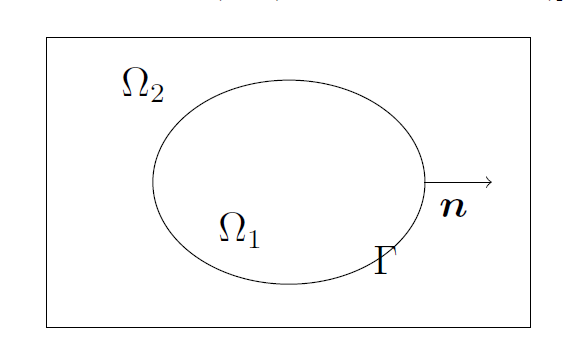}}
  \centerline{$\textbf{Fig. 1.}$ Domain $\Omega$, its subdomains $\Omega_1,\Omega_2$, and interface $\Gamma$.}
\end{figure*}

\subsection{Notations and assumptions}
For a bounded domain $D$, we employ the standard Sobolev space $\bm{H}^m(D)$ with norm $\|\cdot\|_{m,D}$.
When $m=0$, we write $\bm{L}^2(D)$ instead of $\bm{H}^0(D)$ with inner product $(\cdot,\cdot)_D$. For an integer $k\geq 0$, $P_k(D)$ and $\bm{P}_k(D)$ denote the set of all scalar-valued and vector-valued polynomials on domain $D$ with degree less than or equal to $k$, respectively. For a bounded open set $D=D_1\cup D_2$,
 we denote by $\bm{H}^m(D_1\cup D_2)$ the Sobolev spaces of functions in $D$ such that $\bm{p}|_{D_i}\in \bm{H}^m(D_i)$ with the broken norm
\begin{equation*}
  \|\cdot\|_{m,D_1\cup D_2}=\bigg(\displaystyle\sum_{i=1}^{2}\|\cdot\|^2_{m,D_i}\bigg)^{\frac{1}{2}}.
\end{equation*}


Let $\mathcal{T}_h$ be a triangulation of $\Omega$, generated independently of the location of the interface $\Gamma$. Define $h_K$ as the diameter of $K$, and $h=\underset{K\in\mathcal{T}_h}{\mathrm{max}}h_K$.
 By $G_h=\{K\in\mathcal{T}_h\mid K\cap\Gamma\neq\emptyset\}$, we denote the set of elements that are intersected by the interface. For an element $K\in G_h$, we use $\Gamma_K=\Gamma\cap K$ to represent the part of $\Gamma$ in $K$.
 Define the subdomains $\Omega_{i,h}=\{K\in\mathcal{T}_h\mid K\subseteq\Omega_i ~\text{or} ~K\cap\Gamma\neq\emptyset\},i=1,2$.
 \textcolor[rgb]{0.00,0.00,0.00}{Let $\mathcal{T}_{h,i}$ be the triangulation of $\Omega_{i,h}$ and for every triangle $K\in\mathcal{T}_{h,1}\cap\mathcal{T}_{h,2}$, we have $K\cap\Gamma\neq\emptyset$.}
   Define
   $\mathcal{F}_{h,i}=\{e\subseteq\partial K \mid \forall K\in\mathcal{T}_{h,i}\},~\mathcal{F}_{h,i}^{cut}=\{e_i = e|_{\Omega_i}~|~e\subseteq\partial K,~e\cap\Gamma\neq\emptyset,~\forall K\in G_h\},~i=1,2$ (see Figure 2).
 \textcolor[rgb]{0.00,0.00,0.00}{The letter $C$ or $c$, with or without subscript, denotes a generic constant that may not be the same at different occurrences and is independent of the mesh size and $\alpha_1,~\alpha_2$.}

   We make the following assumptions with respect to the mesh and the interface (see \cite{02Hansbo}).
\begin{enumerate}
  \item  The triangulation is non-degenerate, i.e., there exists a constant $C>0$, such that for all $K\in\mathcal{T}_h$, $h_K/ \rho_K$
$\leq C,$ where $\rho_K$ is the diameter of
the largest ball contained in $K$.
  \item $\Gamma$ intersects with the boundary $\partial K$ of an element $K$ in $G_h$ exactly twice and each
(open) edge at most once.
  \item Let $\Gamma_{K,h}$ be the straight line segment connecting the points of intersection between $\Gamma$ and $\partial K$. We assume that $\Gamma_K$ is a function of length on $\Gamma_{K,h}$, in local coordinates
      \begin{equation*}
       \Gamma_{K,h}=\{(\xi,\eta)\mid0<\xi<|\Gamma_{K,h}|,\eta=0\},
      \end{equation*}
and
\begin{equation*}
  \Gamma_K=\{(\xi,\eta)\mid0<\xi<|\Gamma_{K,h}|,\eta=\delta(\xi)\}.
\end{equation*}
\end{enumerate}

The second assumption and the third assumption  are always fulfilled on sufficiently fine meshes under general assumptions, for example, the curvature of $\Gamma$ is bounded.
\begin{figure*}[htbp]
  \center{
  \includegraphics[width=8cm]{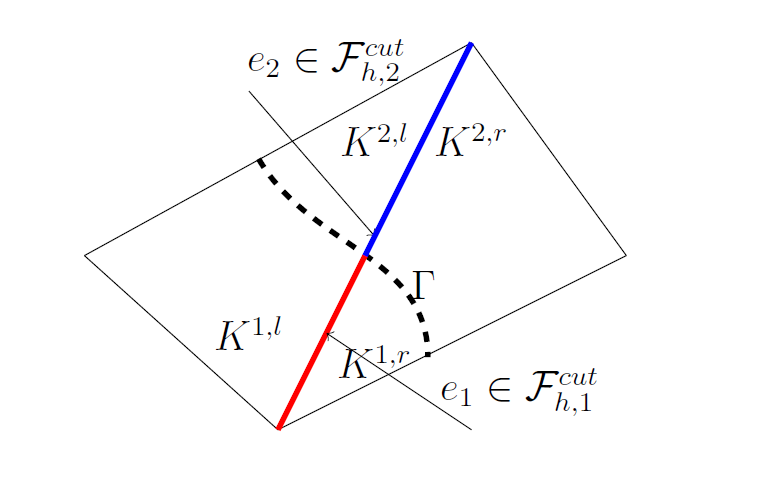}}
  \centerline{$\textbf{Fig. 2.}$ Notations related to adjacent cut elements.}
\end{figure*}

\subsection{The extended finite element method}
Defining
\begin{equation*}
\begin{array}{c}
  \bm{Q}_{h,i}=\{\bm{q_h}\in\bm{H}(\mathrm{div}; \Omega_{i,h})\mid \bm{q_h}|_K\in\bm{P_1}(K),~\forall K\in\mathcal{T}_{h,i}\}, \\
\end{array}
\end{equation*}
\begin{equation*}
 V_{h,i}=\{v_h \in L^2(\Omega_i)\mid v_h|_K\in P_0(K),~\forall K\in\mathcal{T}_{h,i}\}, i=1,2,
\end{equation*}
we introduce the following extended mixed finite element spaces
\begin{eqnarray*}
\bm{Q}_h=\bm{Q}_{h,1}\times \bm{Q}_{h,2}, ~~~
V_h=V_{h,1}\times V_{h,2}.
\end{eqnarray*}

For any function $\varphi$ defined on $K\in G_h$, \textcolor[rgb]{0.00,0.00,0.00}{set} $\{\varphi\}=k^1\varphi_1+k^2\varphi_2$ and $\{\varphi\}_{\ast}=k^2\varphi_1+k^1\varphi_2$ with $k^i=\dfrac{|K_i|}{|K|}$, where $|K|= \mathrm{meas}(K)$, $|K_i|= \mathrm{meas}(K_i)$, $K_i=K\cap\Omega_i,~\varphi_i=\varphi|_{\Omega_i},~i=1,2$. Clearly, $0\leq k^i\leq 1$ and $k^1+k^2=1$.
Recalling the definition of $[\varphi]$, we have $[\varphi\psi]=\{\varphi\}[\psi]+[\varphi]\{\psi\}_{\ast}$.
Before bringing in the mixed formulation for the elliptic interface problem, we need to define the jump of a piecewise smooth, discontinuous function $v$ on the segment $e_i\in \mathcal{F}_{h,i}^{cut}$. For a segment $e_i\in \mathcal{F}_{h,i}^{cut}$ of the whole edge $e\in\mathcal{F}_{h,i}$, let $K^l$ and $K^r$ be the two neighboring cells of edge $e$. Denote $v^l=v|_{K^l}, ~v^r=v|_{K^r}$ and set
\begin{equation*}
  [v]_{e_i}=v^l|_{e_i}-v^r|_{e_i}.
\end{equation*}
 If $e$ is located on $\partial\Omega$, then the jump coincides with the trace.

The discrete variational formulation of problem (\ref{orig pro}) is to find $(\bm{p_h},u_h)\in \bm{Q}_h\times V_h$, such that for any $(\bm{q_h},v_h)\in \bm{Q}_h\times V_h$,
\begin{equation}\label{femwf}
B_h(\bm{p_h},u_h;\bm{q_h},v_h)+\gamma_1J_1(u_h,v_h)+\gamma_2J_2(u_h,v_h)=l_h(\bm{q_h},v_h).
\end{equation}
Here $B_h(\cdot,\cdot;\cdot,\cdot)$ is a bilinear form defined by
\begin{equation*}
  B_h(\bm{p_h},u_h;\bm{q_h},v_h)=a_h(\bm{p_h},\bm{q_h})+b_h(\bm{q_h},u_h)-b_h(\bm{p_h},v_h),
\end{equation*}
where
\begin{align*}
a_h(\bm{p_h},\bm{q_h})=&\int_{\Omega_1\cup\Omega_2}\alpha^{-1}\bm{p_h}\cdot\bm{q_h}dx+\int_{\Omega_1\cup\Omega_2}\alpha^{-1}\mathrm{div}\bm{p_h}\mathrm{div}\bm{q_h}dx\\
&+\gamma h_K^{-1}\int_{\Gamma}\alpha_{\mathrm{min}}^{-1}[\bm{p_h}\cdot\bm{n}][\bm{q_h}\cdot\bm{n}]ds,\\
b_h(\bm{q_h},u_h)=&\int_{\Omega_1\cup\Omega_2}u_h\mathrm{div}\bm{q_h}dx-\int_{\Gamma}\{u_h\}[\bm{q_h}\cdot\bm{n}]ds, \end{align*}
and $l_h(\cdot,\cdot)$ is a linear functional defined by
\begin{align*}
l_h(\bm{q_h},v_h)=&\int_{\Omega}f v_h dx-\int_{\Omega_1\cup\Omega_2}\alpha^{-1}f\mathrm{div}\bm{q_h}dx.
\end{align*}
The stabilization terms are defined as
\begin{align*}
J_1(u_h,v_h)=&\displaystyle{\sum_{i=1}^2}\displaystyle{\sum_{e_i\in\mathcal{F}_{h,i}^{cut}}}\int_{e_i} \alpha_{\mathrm{min}}h[u_h]_{e_i}[v_h]_{e_i}ds,\\
J_2(u_h,v_h)=&\int_{\Gamma} \alpha_{\mathrm{min}}h_K[u_h][v_h]ds,
\end{align*}
 where $\alpha_{\mathrm{min}}$=$\mathrm{min}$$\{\alpha_1$,$\alpha_2\}$, $\gamma_1, \gamma_2, \gamma>0$ are stabilisation parameters independent of $h$.
\begin{rmk}
The third term in $a_h(\cdot,\cdot)$ is a standard penalty term in Nitsche's method \cite{02Hansbo}. The two second terms in $a_h(\cdot,\cdot)$ and $l_h(\cdot,\cdot)$ derived from the first equation in (\ref{orig pro}) are added to guarantee the coercivity of $a_h(\cdot,\cdot)$. The terms $J_1(u_h,v_h)$ and $J_2(u_h,v_h)$ are used to ensure the discrete inf-sup stability of the method.
\end{rmk}



\section{Analysis of the scheme}
Firstly, for any $\bm{p_h}\in\bm{Q}_h,~v_h\in V_h$, we define the following mesh dependent norms,
\begin{align*}
&\|\{v_h\}\|_{{\frac{1}{2}},h,\Gamma}=\bigg(\underset{K\in G_h}{\sum}h_K^{-1}\|\{v_h\}\|_{0,\Gamma_K}^2\bigg)^{\frac{1}{2}},\\
&\|[\bm{p_h}\cdot\bm{n}]\|_{{-\frac{1}{2}},h,\Gamma}=\bigg(\underset{K\in G_h}{\sum}h_K\|[\bm{p_h}\cdot\bm{n}]\|_{0,\Gamma_K}^2\bigg)^{\frac{1}{2}},\\
 &\|\bm{p_h}\|_h=\bigg(\|\alpha^{-\frac{1}{2}}\bm{p_h}\|_{0,\Omega_1\cup\Omega_2}^2
  +\|\alpha^{-\frac{1}{2}}\mathrm{div}\bm{p_h}\|_{0,\Omega_1\cup\Omega_2}^2
  +h^{-2}\|\alpha_{\mathrm{min}}^{-\frac{1}{2}}[\bm{p_h}\cdot\bm{n}]\|_{{-\frac{1}{2}},h,\Gamma}^2\bigg)^{\frac{1}{2}},\\
  &\|v_h\|_{\star}=\bigg(\|\alpha^{\frac{1}{2}}v_h\|_{0,\Omega_1\cup\Omega_2}^2+h_K^2\|\alpha_{\mathrm{min}}^{\frac{1}{2}}\{v_h\}\|_{{\frac{1}{2}},h,\Gamma}^2\bigg)^{\frac{1}{2}},\\
 &\|(\bm{p_h},v_h)\|=\bigg(\|\bm{p_h}\|_h^2+\|v_h\|_{\star}^2\bigg)^{\frac{1}{2}},\\
  &\|(\bm{p_h},v_h)\|_h=\bigg(\|\bm{p_h}\|_h^2+\|\alpha^{\frac{1}{2}}v_h\|_{0,\Omega_1\cup\Omega_2}^2+J_1(v_h,v_h)+J_2(v_h,v_h)\bigg)^{\frac{1}{2}}.
  \end{align*}

Since $\bm{Q}_h\nsubseteq\bm{Q},~V_h\nsubseteq V$,
the finite element formulation (\ref{femwf}) is not consistent. Therefore, we have the following weak consistent relation.
  \begin{lem}\label{consistency}
  Let $(\bm{p},u)\in \bm{H}(\mathrm{div};\Omega)\times H^1(\Omega)$, $(\bm{p_h},u_h)\in \bm{Q}_h\times V_h$ be the pairs of solutions of the problem (\ref{mixorig}) and (\ref{femwf}), then
\begin{equation}
  B_h(\bm{p}-\bm{p_h},u-u_h;\bm{q_h},v_h)=\gamma_1J_1(u_h,v_h)+\gamma_2J_2(u_h,v_h),~~\forall (\bm{q_h},v_h)\in \bm{Q}_h\times V_h.
\end{equation}
\end{lem}
\begin{proof}
  By the definition of $B_h(\cdot,\cdot;\cdot,\cdot)$, Green's formula, $(\ref{orig pro})$ and (\ref{femwf}),
 \begin{align*}
B_h(\bm{p}-\bm{p_h},u-u_h;\bm{q_h},v_h)
=&B_h(\bm{p},u;\bm{q_h},v_h)-B_h(\bm{p_h},u_h;\bm{q_h},v_h)\\
=&\int_{\Omega_1\cup\Omega_2}\alpha^{-1}\bm{p}\cdot\bm{q_h}dx+\int_{\Omega_1\cup\Omega_2}\alpha^{-1}\mathrm{div}\bm{p}\mathrm{div}\bm{q_h}dx\\
&+\int_{\Omega_1\cup\Omega_2}u\mathrm{div}\bm{q_h}dx-\int_{\Gamma}u[\bm{q_h}\cdot\bm{n}]ds\\
&-\int_{\Omega_1\cup\Omega_2}v_h\mathrm{div}\bm{p}dx
+\int_{\Omega_1\cup\Omega_2}\alpha^{-1}f\mathrm{div}\bm{q_h}dx\\
&-\int_{\Omega}fv_hdx+\gamma_1J_1(u_h,v_h)+\gamma_2J_2(u_h,v_h)\\
=&\gamma_1J_1(u_h,v_h)+\gamma_2J_2(u_h,v_h).
\end{align*}
The proof is completed.
\end{proof}

\subsection{Continuity analysis}
\label{sec:2}
\begin{lem}\label{inve}
  For $v_h\in V_h$, there exists a positive constant $C_I$, such that the following inequality holds
  \begin{equation}\label{inverse}
    h_K^2\|\alpha_{\mathrm{min}}^{\frac{1}{2}}\{v_h\}\|_{\frac{1}{2},h,\Gamma}^2\leq C_I\|\alpha^{\frac{1}{2}}v_h\|_{0,\Omega_1\cup\Omega_2}^2.
  \end{equation}
\end{lem}
\begin{proof}
Since $v_h\in V_h$ is a piecewise constant on $K_i$ and $|\Gamma_K|\leq h_K, |K_i|\leq h_K^2, |K|\geq ch_K^2$, we have
  \begin{align*}
    h_K^2\|\alpha_{\mathrm{min}}^{\frac{1}{2}}\{v_h\}\|_{\frac{1}{2},h,\Gamma}^2
    &\leq 2\displaystyle{\sum_{i=1}^2}\displaystyle{\sum_{K\in G_h}}h_K\alpha_{\mathrm{min}}\|v_h|_{\Omega_i}\|_{0,\Gamma_K}^2|k^i|^2\\
    &=2\displaystyle{\sum_{i=1}^2}\displaystyle{\sum_{K\in G_h}}\alpha_{\mathrm{min}}\|v_h|_{\Omega_i}\|_{0,K_i}^2\dfrac{h_K|\Gamma_K|}{|K_i|}\dfrac{|K_i|^2}{|K|^2}\\
    &\leq C_I\|\alpha^{\frac{1}{2}}v_h\|_{0,\Omega_1\cup\Omega_2}^2.
  \end{align*}
  Thus the desired result is obtained.
 \end{proof}
 \textcolor[rgb]{0.00,0.00,0.00}{
 From the definitions of $\|(\bm{q_h},v_h)\|$ and $\|(\bm{q_h},v_h)\|_h$, in light of Lemma \ref{inve}, we can derive the following conclusion.}
\begin{lem}\label{fskz}
  For any $(\bm{q_h},v_h)\in \bm{Q}_h\times V_h,$ there exists a constant $C>0$, such that
    $$\|(\bm{q_h},v_h)\|\leq C\|(\bm{q_h},v_h)\|_h. $$
\end{lem}
By the Cauchy-Schwarz inequality and Lemma \ref{fskz}, the following result is obvious.
\begin{lem}\label{continu}
For all $(\bm{p_h},u_h),~(\bm{q}_h,v_h)\in \bm{Q}_h\times V_h$, there holds
  \begin{equation}\label{bilicon}
   L(\bm{p_h},u_h;\bm{q_h},v_h)\leq C\|(\bm{p_h},u_h)\|_h\|(\bm{q_h},v_h)\|_h,
  \end{equation}
   where $L(\bm{p_h},u_h;\bm{q_h},v_h)=  B_h(\bm{p_h},u_h;\bm{q_h},v_h)+\gamma_1J_1(u_h,v_h)+\gamma_2J_2(u_h,v_h)$.
\end{lem}
\begin{proof}
Using the Cauchy-Schwarz inequality, we obtain
 \begin{align}\label{c1}
a_h(\bm{p_h},\bm{q_h})=&\int_{\Omega_1\cup\Omega_2}\alpha^{-1}\bm{p_h}\cdot\bm{q_h}dx+\int_{\Omega_1\cup\Omega_2}\alpha^{-1}\mathrm{div}\bm{p_h}\mathrm{div}\bm{q_h}dx\notag\\
&+\gamma h_K^{-1}\int_{\Gamma}\alpha_{\mathrm{min}}^{-1}[\bm{p_h}\cdot\bm{n}][\bm{q_h}\cdot\bm{n}]ds\notag\\
\leq& \|\alpha^{-\frac{1}{2}}\bm{p_h}\|_{0,\Omega_1\cup\Omega_2}\|\alpha^{-\frac{1}{2}}\bm{q_h}\|_{0,\Omega_1\cup\Omega_2}
+\|\alpha^{-\frac{1}{2}}\mathrm{div}\bm{p_h}\|_{0,\Omega_1\cup\Omega_2}\|\alpha^{-\frac{1}{2}}\mathrm{div}\bm{q_h}\|_{0,\Omega_1\cup\Omega_2}\notag\\
&+\gamma h_K^{-1}\|\alpha_{\mathrm{min}}^{-\frac{1}{2}}[\bm{p_h}\cdot\bm{n}]\|_{-\frac{1}{2},h,\Gamma}h_K^{-1}\|\alpha_{\mathrm{min}}^{-\frac{1}{2}}[\bm{q_h}\cdot\bm{n}]\|_{-\frac{1}{2},h,\Gamma}\notag\\
\leq& C\|\bm{p_h}\|_h\|\bm{q_h}\|_h,
\end{align}
and
\begin{align}\label{c2}
  b_h(\bm{q_h},u_h)=&\int_{\Omega_1\cup\Omega_2}u_h\mathrm{div}\bm{q_h}dx-\int_{\Gamma}\{u_h\}[\bm{q_h}\cdot\bm{n}]ds\notag\\
  \leq&\|\alpha^{-\frac{1}{2}}\mathrm{div}\bm{q_h}\|_{0,\Omega_1\cup\Omega_2}\|\alpha^{\frac{1}{2}}u_h\|_{0,\Omega_1\cup\Omega_2}\notag\\
  &+h_K^{-1}\|\alpha_{\mathrm{min}}^{-\frac{1}{2}}[\bm{q_h}\cdot\bm{n}]\|_{-\frac{1}{2},h,\Gamma}h_K\|\alpha_{\mathrm{min}}^{\frac{1}{2}}\{u_h\}\|_{\frac{1}{2},h,\Gamma}\notag\\
  \leq& \|\bm{q_h}\|_h\|u_h\|_{\star}.
\end{align}
Similarly, we can get
\begin{equation}\label{c3}
   b_h(\bm{p_h},v_h)\leq \|\bm{p_h}\|_h\|v_h\|_{\star}.
\end{equation}
Combining inequalities (\ref{c1})-(\ref{c3}) with Lemma \ref{fskz}, we obtain (\ref{bilicon}).
\end{proof}

\subsection{Stability analysis}
\label{sec:2}
In this part, we are devoted to the stability analysis of the method. Firstly, we need a variant of the trace inequality on the reference element (see \cite{02Hansbo}).

 \begin{lem}\label{tracedl}
   Map a triangle $K$ onto the reference triangle $\tilde{K}$ by an affine map and denote by $\tilde{\Gamma}_{\tilde{K}}$ the corresponding
   image of ~$\Gamma_K$. Under the three assumptions in Section 2, there exists a positive constant $C$, depending on the shape of $\Gamma$ but independent of the
   mesh, such that
   \begin{equation}\label{traceinq}
     \|w\|_{0,\tilde{\Gamma}_{\tilde{K}}}\leq C\|w\|_{0,\tilde{K}}\|w\|_{1,\tilde{K}} \quad\forall w\in H^1(\tilde{K}).
   \end{equation}
  \end{lem}
     Then the following results holds.

\begin{lem}\label{pistable}
For any $K\in\mathcal{T}_h$, let $\pi_K: \bm{H}^1(K)\rightarrow \mathbb{BDM}_1(K)$ be the lowest order Brezzi-Douglas-Marini                                                                                                                                $(\mathbb{BDM}_1)$ interpolation operator defined on $K$ (see \cite{1985BDM}). Introducing operators $\pi_h^i: \bm{H}^1(\Omega_{i,h})\rightarrow                                                                                                              \bm{Q}_{h,i}$ with $\pi_h^i|_K=\pi_K$, i=1,2, then there exist positive constants $C_1,~C_2$ such that
\begin{equation}\label{pi1s}
  |\pi_h\bm{p_u}|_{1,K}\leq C_1|\bm{p_u}|_{1,K}, \quad\forall \bm{p_u}\in\bm{H}^1(K),
\end{equation}
and
\begin{equation}\label{pihs}
  \|\pi_h\bm{p_u}\|_h\leq C_2\alpha_{\mathrm{min}}^{-\frac{1}{2}}\|\bm{p_u}\|_{1,\Omega}, \quad\forall \bm{p_u}\in\bm{H}^1(\Omega),
\end{equation}
 where $\pi_h|_{\Omega_{i,h}}=\pi_h^i$.
 \end{lem}
\begin{proof}
By scaling arguments, we can obtain
\begin{equation}\label{divsca}
  |\pi_h\bm{p_u}|_{1,K}\leq Ch_K^{-1}|\hat{\pi}_h\bm{\hat{p}_u}|_{1,\hat{K}},
\end{equation}
where $\bm{\hat{p}_u}=P\bm{p_u}$, and $P$ is the Piola's transformation (see III.1.3 in \cite{1991mixfem}).

 Using the definition of $\mathbb{BDM}_1$ interpolation, scaling arguments and the trace theorem on the reference element, we obtain (\ref{pi1s}) easily.
According to the definition of $\mathbb{BDM}_1$ interpolation and scaling arguments, we get
\begin{equation}\label{lsl2}
  \|\pi_K\bm{p_u}\|_{0,K}^2\leq C\bigg(\displaystyle{\sum_{j=1}^3}|\int_{e_j}\bm{p_u}\cdot\bm{n}ds|^2+
  \displaystyle{\sum_{j=1}^3}h^{-2}|\int_{e_j}\bm{p_u}\cdot\bm{n}sds|^2\bigg).
\end{equation}
Combining $(\ref{lsl2})$ with the Cauchy-Schwarz inequality and using the trace theorem, we deduce
\begin{align}
 \|\alpha_i^{-\frac{1}{2}}\pi_K\bm{p_u}\|_{0,K}^2\notag
  \leq& C\alpha_i^{-1}\bigg(\displaystyle{\sum_{j=1}^3}|\int_{e_j}\bm{p_u}\cdot\bm{n}ds|^2
  +\displaystyle{\sum_{j=1}^3}h^{-2}|\int_{e_j}\bm{p_u}\cdot\bm{n}sds|^2\bigg)\notag\\
  \leq&C\alpha_i^{-1}h_K\displaystyle{\sum_{j=1}^3}\|\bm{p_u}\cdot\bm{n}\|_{0,e_j}^2\notag\\
  \leq&C\alpha_i^{-1}\bigg(\|\bm{p_u}\|_{0,K}^2+h_K^2|\bm{p_u}|_{1,K}^2\bigg)\label{pi1}.
 \end{align}
Summing over the triangles $K\in\mathcal{T}_h$, we get
\begin{equation}
  \|\alpha^{-\frac{1}{2}}\pi_h\bm{p_u}\|_{0,\Omega_1\cup\Omega_2}^2 \leq C\alpha_{\mathrm{min}}^{-1}\|\bm{p_u}\|_{1,\Omega}^2\label{pi2}.
\end{equation}

Assuming that $Q_K$ is the $L^2$-projection on $L^2(K)$, using the commuting property (see Proposition 3.7 in \cite{1991mixfem}) on element $K$
\begin{equation}\label{bdmjh}
  \mathrm{div}\pi_K\bm{p_u}=Q_K\mathrm{div}\bm{p_u},
\end{equation}
  and the stability of  the $L^2$-projection $Q_K$, we derive
 \begin{align}\label{pih2}
    \|\alpha^{-\frac{1}{2}}\mathrm{div}\pi_h\bm{p_u}\|_{0,\Omega_1\cup\Omega_2}^2\leq&\displaystyle{\sum_{i=1}^2}\displaystyle{\sum_{K\in\mathcal{T}_{h,i}}}
  \|\alpha_i^{-\frac{1}{2}}\mathrm{div}\pi_K\bm{p_u}\|_{0,K}^2\notag\\
  =&\displaystyle{\sum_{i=1}^2}\displaystyle{\sum_{K\in\mathcal{T}_{h,i}}}\|\alpha_i^{-\frac{1}{2}}Q_K\mathrm{div}\bm{p_u}\|_{0,K}^2\notag\\
  \leq&\displaystyle{\sum_{i=1}^2}\displaystyle{\sum_{K\in\mathcal{T}_{h,i}}}\|\alpha_i^{-\frac{1}{2}}\mathrm{div}\bm{p_u}\|_{0,K}^2\notag\\
  \leq& C\alpha_{\mathrm{min}}^{-1}\|\bm{p_u}\|_{1,\Omega}^2.
  \end{align}

Using the fact $[\bm{p_u}\cdot\bm{n}]|_{\Gamma}=0$, the trace inequality (\ref{traceinq}), the standard $\mathbb{BDM}_1$ interpolation estimates and (\ref{pi1s}), we obtain
 \begin{align}\label{pih3}
   h_K^{-2}\|\alpha_{\mathrm{min}}^{-\frac{1}{2}}[\pi_h\bm{p_u}\cdot\bm{n}]\|_{-\frac{1}{2},h,\Gamma}^2=&\displaystyle{\sum_{K\in G_h}}h_K^{-1}\alpha_{\mathrm{min}}^{-1}\|[(\pi_K\bm{p_u}-\bm{p_u})\cdot\bm{n}]\|_{0,\Gamma_K}^2\notag\\
   \leq&C\displaystyle{\sum_{K\in G_h}}h_K^{-1}\alpha_{\mathrm{min}}^{-1}\bigg(h_K^{-1}\|\pi_K\bm{p_u}-\bm{p_u}\|_{0,K}^2
   +h_K|\pi_K\bm{p_u}-\bm{p_u}|_{1,K}^2\bigg)\notag\\
   \leq&C\displaystyle{\sum_{K\in G_h}}\alpha_{\mathrm{min}}^{-1}\bigg(h_K^{-2}h_K^2\|\bm{p_u}\|_{1,K}^2+|\pi_K\bm{p_u}-\bm{p_u}|_{1,K}^2\bigg)\notag\\
   \leq& C \alpha_{\mathrm{min}}^{-1}\|\bm{p_u}\|_{1,\Omega}^2,
    \end{align}
   \textcolor[rgb]{0.00,0.00,0.00}{which together with (\ref{pi2}) and (\ref{pih2}) yields the inequality (\ref{pihs}).}
\end{proof}

\begin{thm}\label{lbb}
         Let $(\bm{p_h},u_h)\in\bm{Q}_h\times V_h$, for sufficiently small $h$, there exists a positive constant $C_s$ which is independent of $h$ and how the interface intersects the triangulation,
          such that
         \begin{equation}\label{lbbineq}
           \underset{(\bm{q_h},v_h)\in\bm{Q}_h\times V_h}{\mathrm{sup}}
           \dfrac{L(\bm{p_h},u_h;\bm{q_h},v_h)}{\|(\bm{q_h},v_h)\|_h}
           \geq C_s\|(\bm{p_h},u_h)\|_h,
         \end{equation}
         where $L(\bm{p_h},u_h;\bm{q_h},v_h)=B_h(\bm{p_h},u_h;\bm{q_h},v_h)
           +\gamma_1J_1(u_h,v_h)+\gamma_2J_2(u_h,v_h)$.
       \end{thm}
\begin{proof}
 Noting the definition of $a_h(\bm{p_h},\bm{p_h})$, we get
 \begin{align*}
 a_h(\bm{p_h},\bm{p_h})=&\int_{\Omega_1\cup\Omega_2}\alpha^{-1}|\bm{p_h}|^2dx+\int_{\Omega_1\cup\Omega_2}\alpha^{-1}|\mathrm{div}\bm{p_h}|^2dx
+\gamma h_K^{-1}\int_{\Gamma}\alpha_{\mathrm{min}}^{-1}|[\bm{p_h}\cdot\bm{n}]|^2ds\\
\geq& C_p\|\bm{p_h}\|_h^2,
 \end{align*}
 \textcolor[rgb]{0.00,0.00,0.00}{
which implies}
 \begin{equation*}
   B_h(\bm{p_h},u_h;\bm{p_h},u_h)\geq C_p\|\bm{p_h}\|_h^2.
 \end{equation*}

 By Lemma 11.2.3 in \cite{bookbrener}, we know that for any $u_h\in V_h$, there exists a function $\bm{p_u}\in \bm{H}^1(\Omega)$, such that $\mathrm{div}\bm{p_u}=\alpha_{\mathrm{min}}u_h,$ and $\|\bm{p_u}\|_{1,\Omega}\leq c\alpha_{\mathrm{min}}\|u_h\|_{0,\Omega}$. Taking $(\bm{q_h},v_h)=(\pi_h\bm{p_u},0)$, where $\pi_h$ is the interpolation operator defined in Lemma \ref{pistable}, leads to
\begin{equation}\label{bh1}
  B_h(\bm{p_h},u_h;\pi_h\bm{p_u},0)=a_h(\bm{p_h},\pi_h\bm{p_u})+b_h(\pi_h\bm{p_u}-\bm{p_u},u_h)+\alpha_{\mathrm{min}}\|u_h\|_{0,\Omega_1\cup\Omega_2}^2.
\end{equation}
Applying the Cauchy-Schwarz inequality and the stability of $\pi_h$ (\ref{pihs}), we derive
\begin{align}\label{bh2}
  a_h(\bm{p_h},\pi_h\bm{p_u})\geq& -C\|\bm{p_h}\|_h\|\pi_h\bm{p_u}\|_h
  \geq -C\|\bm{p_h}\|_h\alpha_{\mathrm{min}}^{-\frac{1}{2}}\|\bm{p_u}\|_{1,\Omega}\notag\\
  \geq&-C_a\|\bm{p_h}\|_h\alpha_{\mathrm{min}}^{\frac{1}{2}}\|u_h\|_{0,\Omega}\notag\\
  \geq& -C_a\|\bm{p_h}\|_h\|\alpha^{\frac{1}{2}}u_h\|_{0,\Omega_1\cup\Omega_2}.
\end{align}
Combining Green's formula with the definition of the interpolation operator $\pi_h$, we obtain
\begin{align}\label{bh3}
  b_h(\pi_h\bm{p_u}-\bm{p_u},u_h)=&\int_{\Omega_1\cup\Omega_2}u_h\mathrm{div}(\pi_h\bm{p_u}-\bm{p_u})dx
  -\int_{\Gamma}\{u_h\}[(\pi_h\bm{p_u}-\bm{p_u})\cdot\bm{n}]ds\notag\\
  =&\displaystyle{\sum_{i=1}^2}\displaystyle{\sum_{e_i\in\mathcal{F}_{h,i}^{cut}}}\int_{e_i}[u_h]_{e_i}(\pi_h\bm{p_u}-\bm{p_u})\cdot\bm{n}ds
  +\int_{\Gamma}\{(\pi_h\bm{p_u}-\bm{p_u})\cdot\bm{n}\}_*[u_h]ds\notag\\
  =&I+II.
\end{align}
Using the Cauchy-Schwarz inequality, the trace inequality, the approximability of $\pi_h$ and $(\ref{pi1s})$, we get
\begin{align}\label{bh31}
  I\leq &\displaystyle{\sum_{i=1}^2}\displaystyle{\sum_{e_i\in\mathcal{F}_{h,i}^{cut}}}\|[u_h]_{e_i}\|_{0,e_i}\|(\pi_h\bm{p_u}-\bm{p_u})\cdot\bm{n}\|_{0,e_i}\notag\\
  \leq&\bigg(\displaystyle{\sum_{i=1}^2}\displaystyle{\sum_{e_i\in\mathcal{F}_{h,i}^{cut}}}\|[u_h]_{e_i}\|_{0,e_i}^2\bigg)^{\frac{1}{2}}
  \bigg(\displaystyle{\sum_{i=1}^2}\displaystyle{\sum_{e\in\mathcal{F}_{h,i}}}\|(\pi_h\bm{p_u}-\bm{p_u})\cdot\bm{n}\|_{0,e}^2\bigg)^{\frac{1}{2}}\notag\\
  \leq&C_1\bigg(\displaystyle{\sum_{i=1}^2}\displaystyle{\sum_{e_i\in\mathcal{F}_{h,i}^{cut}}}\|[u_h]_{e_i}\|_{0,e_i}^2\bigg)^{\frac{1}{2}}
  \bigg(\displaystyle{\sum_{i=1}^2}\displaystyle{\sum_{K\in\mathcal{T}_{h,i}}}\big(h_{K}^{-1}\|\pi_K\bm{p_u}-\bm{p_u}\|_{0,K}^2
  +h_{K}|\pi_K\bm{p_u}-\bm{p_u}|_{1,K}^2\big)\bigg)^{\frac{1}{2}}\notag\\
  \leq&C_1\bigg(\displaystyle{\sum_{i=1}^2}\displaystyle{\sum_{e_i\in\mathcal{F}_{h,i}^{cut}}}\|[u_h]_{e_i}\|_{0,e_i}^2\bigg)^{\frac{1}{2}}
  \bigg(\displaystyle{\sum_{i=1}^2}\displaystyle{\sum_{K\in\mathcal{T}_{h,i}}}h_{K}|\bm{p_u}|_{1,K}^2\bigg)^{\frac{1}{2}}\notag\\
  \leq& C_1J_1(u_h,u_h)^{\fr{1}{2}}\|\alpha^{\fr{1}{2}}u_h\|_{0,\Omega_1\cup\Omega_2},
  \end{align}
  where $e$ is the whole edge containing the cut segment $e_i$. Similarly, we obtain
\begin{align}\label{bh32}
  II\leq&\displaystyle{\sum_{K\in G_h}}\|[u_h]\|_{0,\Gamma_K}\|\{(\pi_K\bm{p_u}-\bm{p_u})\cdot\bm{n}\}_{*}\|_{0,\Gamma_K}\notag\\
  \leq&C_2\displaystyle{\sum_{K\in G_h}}\|[u_h]\|_{0,\Gamma_K}h_K^{\frac{1}{2}}\|\bm{p_u}\|_{1,K}\notag\\
  \leq&C_2\bigg(\displaystyle{\sum_{K\in G_h}}\alpha_{\mathrm{min}}h_K\|[u_h]\|_{0,\Gamma_K}^2\bigg)^{\fr{1}{2}}
  \bigg(\displaystyle{\sum_{K\in G_h}}\alpha_{\mathrm{min}}^{-1}\|\bm{p_u}\|_{1,K}^2\bigg)^{\fr{1}{2}}\notag\\
  \leq&C_2J_2(u_h,u_h)^{\fr{1}{2}}\|\alpha^{\fr{1}{2}}u_h\|_{0,\Omega_1\cup\Omega_2}.
\end{align}
Thus, using (\ref{bh1})-(\ref{bh32}) and the arithmetic-geometric inequality, we deduce
\begin{align*}
   B_h(\bm{p_h},u_h;\pi_h\bm{p_u},0)=&a_h(\bm{p_h},\pi_h\bm{p_u})+b_h(\pi_h\bm{p_u}-\bm{p_u},u_h)+\alpha_{\mathrm{min}}\|u_h\|_{0,\Omega}^2\\
   \geq& -C_a\|\bm{p_h}\|_h\|\alpha^{\fr{1}{2}}u_h\|_{0,\Omega_1\cup\Omega_2}
   -C_1J_1(u_h,u_h)^{\fr{1}{2}}\|\alpha^{\fr{1}{2}}u_h\|_{0,\Omega_1\cup\Omega_2}\\
   &-C_2J_2(u_h,u_h)^{\fr{1}{2}}\|\alpha^{\fr{1}{2}}u_h\|_{0,\Omega_1\cup\Omega_2}+\alpha_{\mathrm{min}}\|u_h\|_{0,\Omega}^2\\
   \geq&\dfrac{\alpha_{\mathrm{min}}}{\alpha_{\mathrm{max}}}\|\alpha^{\fr{1}{2}}u_h\|_{0,\Omega_1\cup\Omega_2}^2-\bigg(C_a\|\bm{p_h}\|_h\\
   &+C_1J_1(u_h,u_h)^{\fr{1}{2}}+C_2J_2(u_h,u_h)^{\fr{1}{2}}\bigg)\|\alpha^{\fr{1}{2}}u_h\|_{0,\Omega_1\cup\Omega_2}\\
   \geq&\dfrac{\alpha_{\mathrm{min}}}{2\alpha_{\mathrm{max}}}\|\alpha^{\fr{1}{2}}u_h\|_{0,\Omega_1\cup\Omega_2}^2
   -C_a^2\dfrac{\alpha_{\mathrm{max}}}{2\alpha_{\mathrm{min}}}\|\bm{p_h}\|_h^2\\
   &-C_1^2\dfrac{\alpha_{\mathrm{max}}}{2\alpha_{\mathrm{min}}}J_1(u_h,u_h)-C_2^2\dfrac{\alpha_{\mathrm{max}}}{2\alpha_{\mathrm{min}}}J_2(u_h,u_h).
\end{align*}

Then the following inequality holds
\begin{align*}
   &B_h(\bm{p_h},u_h;\bm{p_h}+\varepsilon\pi_h\bm{p_u},u_h)+\gamma_1J_1(u_h,u_h)+\gamma_2J_2(u_h,u_h)\\
  =&B_h(\bm{p_h},u_h;\bm{p_h},u_h)
  +B_h(\bm{p_h},u_h;\varepsilon\pi_h\bm{p_u},0)+\gamma_1J_1(u_h,u_h)+\gamma_2J_2(u_h,u_h)\\
   \geq&C_p\|\bm{p_h}\|_h^2+\varepsilon\bigg(\dfrac{\alpha_{\mathrm{min}}}{2\alpha_{\mathrm{max}}}\|\alpha^{\fr{1}{2}}u_h\|_{0,\Omega_1\cup\Omega_2}^2
   -C_a^2\dfrac{\alpha_{\mathrm{max}}}{2\alpha_{\mathrm{min}}}\|\bm{p_h}\|_h^2\\
   &-C_1^2\dfrac{\alpha_{\mathrm{max}}}{2\alpha_{\mathrm{min}}}J_1(u_h,u_h)-C_2^2\dfrac{\alpha_{\mathrm{max}}}{2\alpha_{\mathrm{min}}}J_2(u_h,u_h)\bigg)+\gamma_1J_1(u_h,u_h)+\gamma_2J_2(u_h,u_h)\\
   \geq&\bigg(C_p-\varepsilon C_a^2\dfrac{\alpha_{\mathrm{max}}}{2\alpha_{\mathrm{min}}}\bigg)\|\bm{p_h}\|_h^2
   +\varepsilon\dfrac{\alpha_{\mathrm{min}}}{2\alpha_{\mathrm{max}}}\|\alpha^{\fr{1}{2}}u_h\|_{0,\Omega_1\cup\Omega_2}^2\\
   &+\bigg(\gamma_1-\varepsilon C_1^2\dfrac{\alpha_{\mathrm{max}}}{2\alpha_{\mathrm{min}}}\bigg)J_1(u_h,u_h)+\bigg(\gamma_2-\varepsilon C_2^2\dfrac{\alpha_{\mathrm{max}}}{2\alpha_{\mathrm{min}}}\bigg)J_2(u_h,u_h),
\end{align*}
\textcolor[rgb]{0.00,0.00,0.00}{setting $\varepsilon=\dfrac{C_p}{C_a^2}\dfrac{\alpha_{\mathrm{min}}}{\alpha_{\mathrm{max}}}$, $\gamma_1 =\dfrac{C_pC_1^2}{C_a^2} $, $\gamma_2 =\dfrac{C_pC_2^2}{C_a^2}$, we derive that
\begin{equation}
  B_h(\bm{p_h},u_h;\bm{p_h}+\varepsilon\pi_h\bm{p_u},u_h)+\gamma_1J_1(u_h,u_h)+\gamma_2J_2(u_h,u_h)\geq C_s\|(\bm{p_h},u_h)\|_h^2,
\end{equation}
where $C_s = \mathrm{min}\bigg\{\dfrac{C_p}{2}, \dfrac{C_p}{2}\dfrac{C_1^2}{C_a^2}, \dfrac{C_p}{2}\dfrac{C_2^2}{C_a^2},\dfrac{C_p}{2C_a^2}\dfrac{\alpha_{\mathrm{min}}^2}{\alpha_{\mathrm{max}}^2} \bigg\}$.}

Now we prove $\|(\bm{p_h}+\varepsilon\pi_h\bm{p_u},u_h)\|_h\leq \|(\bm{p_h},u_h)\|_h$. In view of the definition of $\|\cdot\|_h$, we only need to prove that
$$\|\bm{p_h}+\varepsilon\pi_h\bm{p_u}\|_h\leq C\|(\bm{p_h},u_h)\|_h.$$
Indeed, by the triangle inequality and the stability of $\pi_h$, we have
\begin{align*}
 \|\bm{p_h}+\varepsilon\pi_h\bm{p_u}\|_h^2&\leq 2\bigg(\|\bm{p_h}\|_h^2+\varepsilon^2\|\pi_h\bm{p_u}\|_h^2\bigg)
 \leq 2\bigg(\|\bm{p_h}\|_h^2+C\varepsilon^2\alpha_{\mathrm{min}}^{-1}\|\bm{p_u}\|_{1,\Omega}^2\bigg)\\
 &\leq 2\bigg(\|\bm{p_h}\|_h^2+C\varepsilon^2\alpha_{\mathrm{min}}\|u_h\|_{0,\Omega}^2\bigg)
 \leq  C\|(\bm{p_h},u_h)\|_h^2.
 \end{align*}
That's to say, for any $(\bm{p_h},u_h)\in\bm{Q}_h\times V_h$, there exists $(\bm{p_h}+\varepsilon\pi_h\bm{p_u},u_h)\in\bm{Q}_h\times V_h$, such that
\begin{equation*}
             B_h(\bm{p_h},u_h;\bm{p_h}+\varepsilon\pi_h\bm{p_u},u_h)
           +\gamma_1J_1(u_h,u_h)+\gamma_2J_2(u_h,u_h)\geq C_s\|(\bm{p_h},u_h)\|_h^2,
\end{equation*}
and
$$\|(\bm{p_h}+\varepsilon\pi_h\bm{p_u},u_h)\|_h\leq C\|(\bm{p_h},u_h)\|_h.$$
Then the result $(\ref{lbbineq})$ is obtained.
\end{proof}

According to Lemma \ref{continu} and Theorem \ref{lbb}, by the Babu\v{s}ka Theorem (cf. \cite{1991mixfem}, \cite{1986NS}), the following theorem holds.

 \begin{thm}
      The problem (\ref{femwf}) has a unique solution $(\bm{p_h},u_h)\in\bm{Q}_h\times V_h$.
     \end{thm}

  \section{Approximation properties and optimal convergence}
\label{sec:1}
\textcolor[rgb]{0.00,0.00,0.00}{In order to construct the interpolation operator for the extended finite element space}, firstly we need to introduce a bounded linear extension operator for vector spaces.

       Motivated by the construction of extension operators for functions in Sobolev spaces $H^k(\Omega)$ (\cite{yantuo}, \cite{exten}), we propose a new extension for functions in the space $\bm{H}^1(\mathrm{div};\Omega)$. This new extension plays a crucial role in the subsequent error estimate for the interface problems.
       The following extension theorem can be found in \cite{exten} (Theorem 1, Sect. 5.4).
       \begin{thm}\label{h2yt}
         ($H^2$-extension theorem) Assume that $U$ is a connected bounded domain in $\mathbb{R}^2$ with $C^2$-smooth boundary $\partial U$. Choose a bounded open set $V$ such that $U\subset\subset V$. Then there exists a bounded linear extension operator
         \begin{equation*}
           E: H^2(U)\rightarrow H^2(\mathbb{R}^2),
         \end{equation*}
       such that for any scalar function $u\in H^2(U)$
   \begin{enumerate}
     \item  $Eu=u$ a.e. in $U$, and $Eu$ has a support within $V$;
     \item   $\|Eu\|_{H^2(\mathbb{R}^2)}\leq C\|u\|_{H^2(U)}$ with $C=C(U,V)>0$.
   \end{enumerate}
       \end{thm}

 In order to preserve the similar properties in Theorem \ref{h2yt}, vector fields must be extended in a more delicate way. Given a vector $\bm{u}\in\bm{H}^1(\mathrm{div};U)$, we wish to extend $\bm{u}$ to $\tilde{\bm{u}}\in\bm{H}^1(\mathrm{div};\mathbb{R}^2)$. For a scalar function $p\in H^2(U)$, we have $\bm{rot}~p\in \bm{H}^1(\mathrm{div};U)$. It looks possible to define an $\bm{H}^1(\mathrm{div})$-extension operator $\bm{E}_{\mathrm{div}}$ satisfying the following commuting diagram property (\cite{jhtu})
        \begin{equation}\label{jhtu}
         \bm{E}_{\mathrm{div}}(\bm{rot}~p)=\bm{rot}(Ep),
       \end{equation}
where $\bm{rot}~p=(\dfrac{\partial p}{\partial x_2},-\dfrac{\partial p}{\partial x_1})^T$, $T$ represents the transposition.

If $\bm{E}_{\mathrm{div}}$ satisfies (\ref{jhtu}), it is obvious that the operator $\bm{E}_{\mathrm{div}}$ preserves the $\mathrm{div}$-free property of a rot field in $U$. While for general vector fields, we can make use of (\ref{jhtu}) to construct a general extension operator $\bm{E}_{\mathrm{div}}$.

With the motivation above, now we establish an $\bm{H}^1(\mathrm{div})$-extension operator across the $C^2$-smooth boundary.
       \begin{thm}\label{divth}
         $\big(\bm{H}^1(\mathrm{div})$-extension~theorem\big) Assume that $U$ is a connected bounded domain in $\mathbb{R}^2$ with $C^2$-smooth boundary. Then there exists a bounded linear extension operator
         \begin{equation*}
           \bm{E}_{\mathrm{div}}: \bm{H}^2(U)\cap\bm{H}^1(\mathrm{div};U)\rightarrow \bm{H}^2(\mathbb{R}^2)\cap\bm{H}^1(\mathrm{div};\mathbb{R}^2),
         \end{equation*}
         such that, for each $\bm{u}\in \bm{H}^2(U)\cap\bm{H}^1(\mathrm{div};U)$, it holds that
         \begin{enumerate}
           \item $\bm{E}_{\mathrm{div}}\bm{u}=\bm{u}~~a.e.~in~U;$
           \item $\|\mathrm{div}(\bm{E}_{\mathrm{div}}\bm{u})\|_{H^1(\mathbb{R}^2)}\leq C\|\mathrm{div}\bm{u}\|_{H^1(U)}$,
               with the constant $C$ depending only on $U$;
            \item  $\|\bm{E}_{\mathrm{div}}\bm{u}\|_{\bm{H}^s(\mathbb{R}^2)}\leq c \|\bm{u}\|_{\bm{H}^s(U)},~s=1,2$, with the constant $c$ depending only on $U$.
         \end{enumerate}
       \end{thm}
\begin{proof}
          Without loss of generality, for a fixed $\bm{x}^0\in\partial U$, we first suppose that $\partial U$ is flat near $\bm{x}^0$ which is lying in the plane $\{\bm{x}=(x_1,x_2)|x_2=0\}$. Assume that there exists an open ball
         \begin{align*}
          B=\{\bm{x}|~|\bm{x-x^0}|\leq r\},
         \end{align*}
         with the center $\bm{x^0}$ and the radius $r>0$ such that

         \begin{equation}
                  \left\{
         \begin{array}{l}
         B^{+}=B\cap \{x_2\geq 0\}\subset\overline{U},\\
         B^{-}=B\cap\{x_2<0\}\subset \mathrm{R}^2\backslash\overline{U}.
          \end{array}
         \right.
         \end{equation}
         Assuming $p\in C^\infty(\overline{U})$, we obtain an extension $E^Bp$ of $p$ from  $B^{+}$ to $B$ as follows,
         \begin{equation}
                  E^Bp=\tilde{p}(\bm{x})=\left\{
         \begin{array}{lcr}
         p(\bm{x}), &\mathrm{if} \quad\bm{x}\in B^{+},\\
         \displaystyle{\sum_{j=1}^4}\lambda_jp(x_1,-\dfrac{x_2}{j}),&\mathrm{if} \quad\bm{x}\in B^{-},
          \end{array}
         \right.
         \end{equation}
         where $(\lambda_1,\lambda_2,\lambda_3,\lambda_4)$ is the solution of the 4$\times$4 system of linear equations
         \begin{equation}\label{lambd}
           \displaystyle{\sum_{j=1}^4}(-\dfrac{1}{j})^k\lambda_j=1,\quad k=0,1,2,3.
         \end{equation}
         It is readily checked that
        \begin{equation*}
          \tilde{p}\in C^2(\overline{B}).
        \end{equation*}

       Now we define an extension $\bm{E}_{\mathrm{div}}^B(\bm{rot}~p)$ of $\bm{rot}~p$ from $B^{+}$ to $B$ as
       \begin{equation}\label{rot}
                 \bm{E}_{\mathrm{div}}^B(\bm{rot}~p)=\widetilde{\bm{rot}~p}=\left\{
         \begin{array}{lcr}
         \bm{rot}~p, &\mathrm{if} \quad\bm{x}\in B^{+},\\
         \bm{rot}~\tilde{p},&\mathrm{if} \quad\bm{x}\in B^{-},
          \end{array}
         \right.
         \end{equation}
         i.e.,

         \begin{equation}\label{rotp}
                 \widetilde{\bm{rot}~p}=\left\{
         \begin{array}{lcr}
         \left(
         \begin{array}{cc}
         p_{x_2}\\
         -p_{x_1}
        \end{array}
        \right), &\mathrm{if} \quad\bm{x}\in B^{+},\\
        \left(
        \begin{array}{cc}
        \displaystyle{\sum_{j=1}^4}-\dfrac{\lambda_j}{j}p_{x_2}(x_1,-\dfrac{x_2}{j})\\
        -\displaystyle{\sum_{j=1}^4}\lambda_jp_{x_1}(x_1,-\dfrac{x_2}{j})
        \end{array}
        \right),
         &\mathrm{if} \quad\bm{x}\in B^{-},
          \end{array}
         \right.
         \end{equation}
    where $p_{x_i}=\dfrac{\partial p}{\partial x_i}, ~i=1, 2$. We can see (\ref{jhtu}) holds from (\ref{rot}).
         Based on the component form of $\widetilde{\bm{rot}~p}$ in (\ref{rotp}), we derive an extension formula for a vector-valued function
          $\bm{w}=(w^1,w^2)^T\in C^\infty(B^{+})$ as follows
          \begin{equation}\label{wtilte}
                 \tilde{\bm{w}}(\bm{x})=\left(
        \begin{array}{cc}
        \tilde{w}^1(\bm{x})\\
        \tilde{w}^2(\bm{x})
        \end{array}
        \right)=\left\{
         \begin{array}{lcr}
        \bm{w}(\bm{x}), &\mathrm{if} \quad\bm{x}\in B^{+},\\
        \left(
        \begin{array}{cc}
        \displaystyle{\sum_{j=1}^4}-\dfrac{\lambda_j}{j}w^1(x_1,-\dfrac{x_2}{j})\\
        \displaystyle{\sum_{j=1}^4}\lambda_jw^2(x_1,-\dfrac{x_2}{j})
        \end{array}
        \right),
         &\mathrm{if} \quad\bm{x}\in B^{-}.
          \end{array}
         \right.
         \end{equation}
          We claim $\tilde{\bm{w}}\in C^1(B)$ and thus $\mathrm{div}\tilde{\bm{w}}\in C^0(B)$. This can be demonstrated by a detailed calculation. Indeed according to (\ref{lambd}) and (\ref{wtilte}), we have
         \begin{align*}
           \displaystyle{\lim_{x_2\rightarrow 0^+}}\tilde{w}^i(\bm{x})&=\displaystyle{\lim_{x_2\rightarrow 0^-}}\tilde{w}^i(\bm{x}), ~i=1,2,\\
           \displaystyle{\lim_{x_2\rightarrow 0^+}}\tilde{w}^i_{x_j}(\bm{x})&=\displaystyle{\lim_{x_2\rightarrow 0^-}}\tilde{w}^i_{x_j}(\bm{x}), ~i,j=1,2.
         \end{align*}

         With this extension $\tilde{\bm{w}}$ on hand, now we prove
         \begin{equation}\label{div}
          \|\mathrm{div}\tilde{\bm{w}}\|_{H^1(B)}\leq \|\mathrm{div}\bm{w}\|_{H^1(B^+)},
       \end{equation}
       and
       \begin{equation}
        \|\tilde{\bm{w}}\|_{H^s(B)}\leq \|\bm{w}\|_{H^s(B^+)},~s=1,2.
       \end{equation}
       In fact
       \begin{align}
         \int_B|\mathrm{div}\tilde{\bm{w}}(\bm{x})|^2d \bm{x}=&\int_{B^-}\bigg|\displaystyle{\sum_{j=1}^4}-\dfrac{\lambda_j}{j}w^1_{x_1}(x_1,-\dfrac{x_2}{j})
         +\displaystyle{\sum_{j=1}^4}-\dfrac{\lambda_j}{j}w^2_{x_2}(x_1,-\dfrac{x_2}{j})\bigg|^2d \bm{x}\notag\\
         &+\int_{B^+}\big|w^1_{x_1}+w^2_{x_2}\big|^2d\bm{x} .\label{div1}
       \end{align}
       Let $\tilde{x}_2=-\dfrac{x_2}{j}$, $\tilde{\bm{x}}=(x_1,\tilde{x}_2)$, by (\ref{lambd}), we derive
       \begin{align}
         &\int_{B^-}\bigg|\displaystyle{\sum_{j=1}^4}-\dfrac{\lambda_j}{j}w^1_{x_1}(x_1,-\dfrac{x_2}{j})
         +\displaystyle{\sum_{j=1}^4}-\dfrac{\lambda_j}{j}w^2_{x_2}(x_1,-\dfrac{x_2}{j})\bigg|^2d \bm{x}\notag\\
         \leq&\int_{B^+}\bigg|\displaystyle{\sum_{j=1}^4}-\dfrac{\lambda_j}{j}w^1_{x_1}(x_1,\tilde{x}_2)
         +\displaystyle{\sum_{j=1}^4}(-\dfrac{1}{j})^2\lambda_jw^2_{\tilde{x}_2}(x_1,\tilde{x}_2)\bigg|^2jd \tilde{\bm{x}}\notag\\
         \leq& C \int_{B^+}\big|w^1_{x_1}+w^2_{x_2}\big|^2d \bm{x}=C\int_{B^+}|\mathrm{div}\bm{w}|^2d \bm{x}.\label{div2}
       \end{align}
       Combining (\ref{div1}) with (\ref{div2}), we obtain (\ref{div}).
       Similarly, we get
       \begin{align*}
         \int_B|\bm{grad}\big(\mathrm{div}\tilde{\bm{w}}(\bm{x})\big)|^2d \bm{x}=&\displaystyle{\sum_{k=1}^2}\int_{B^+}|w^1_{x_1,x_k}+w^2_{x_2,x_k}|^2d \bm{x}\\
         &+\displaystyle{\sum_{k=1}^2}\int_{B^-}\bigg|\displaystyle{\sum_{j=1}^4}(-1)^k\dfrac{\lambda_j}{j^{\lfloor (k+2)/2\rfloor}}w^1_{x_1,x_k}(x_1,-\dfrac{x_2}{j})\\
         &+\displaystyle{\sum_{j=1}^4}(-1)^k\dfrac{\lambda_j}{j^{\lfloor (k+2)/2\rfloor}}w^2_{x_2,x_k}(x_1,-\dfrac{x_2}{j})\bigg|^2d \bm{x}\\
         \leq& C \displaystyle{\sum_{k=1}^2}\int_{B^+}\big|w^1_{x_1,x_k}+w^2_{x_2,x_k}\big|^2d \bm{x},
       \end{align*}
then we have $\|\mathrm{div}\tilde{\bm{w}}\|_{\bm{H}^1(B)}\leq C\|\mathrm{div}\bm{w}\|_{\bm{H}^1(B^{+})}$.

      Combining the same argument as the above with the Cauchy-schwarz inequality, for $\alpha=(\alpha_1,\alpha_2), ~|\alpha|\leq 2$, we deduce
         \begin{align*}
           \int_B|D^\alpha\tilde{\bm{w}}(\bm{x})|^2d \bm{x}=&\int_{B^+}\bigg(|D^\alpha w^1(\bm{x})|^2+|D^\alpha w^2(\bm{x})|^2\bigg)d \bm{x}\\
           &+\int_{B^-}\bigg(\big|\displaystyle{\sum_{j=1}^4}(-\dfrac{1}{j})^{\alpha_2+1}\lambda_jD^\alpha w^1(x_1,-\dfrac{x_2}{j})|^2\\
           &+\big|\displaystyle{\sum_{j=1}^4}(-\dfrac{1}{j})^{\alpha_2}\lambda_jD^\alpha w^2(x_1,-\dfrac{x_2}{j})\big|^2\bigg)d \bm{x}\\
           \leq& C
           \int_{B^+}|D^\alpha\bm{w}(\bm{x})|^2d \bm{x}.
         \end{align*}
         Thus, we obtain $\|\tilde{\bm{w}}\|_{\bm{H}^s(B)}\leq C\|\bm{w}\|_{\bm{H}^s(B^{+})},~s=1,2$. It means the extension
$\bm{E}_{\mathrm{div}}^B$ has the following properties:
\begin{enumerate}
 \item $\bm{E}_{\mathrm{div}}^B\bm{w}=\bm{w}~~a.e.~in~~ B^{+};$
 \item $\|\mathrm{div}(\bm{E}_{\mathrm{div}}^B\bm{w})\|_{H^1(B)}\leq C\|\mathrm{div}\bm{w}\|_{H^1(B^{+})}$;
 \item  $\|\bm{E}_{\mathrm{div}}^B\bm{w}\|_{\bm{H}^s(B)}\leq c \|\bm{w}\|_{\bm{H}^s(B^{+})},~s=1,2$.
\end{enumerate}

         For the general case, arguing as Theorem 3.2 in \cite{divex},
           we complete the proof.
         \end{proof}

Noting the fact that the interface $\Gamma$ is at least $C^2$-smooth and arguing similarly as that in Theorem \ref{divth}, we can derive the following corollary.
\begin{corollary}\label{omegayt}
  There exist a bounded linear operator for i=1,2,
  \begin{equation}\label{diviex}
    \bm{E}_{\mathrm{div}}^i: \bm{H}^2(\Omega_i)\cap\bm{H}^1(\mathrm{div};\Omega_i)\rightarrow \bm{H}^2(\Omega)\cap\bm{H}^1(\mathrm{div};\Omega),
  \end{equation}
  such that for each $\bm{u}\in \bm{H}^2(\Omega_i)\cap\bm{H}^1(\mathrm{div};\Omega_i)$
\begin{enumerate}
  \item $\bm{E}_{\mathrm{div}}^i\bm{u}=\bm{u}$~~a.e.~in~$\Omega_i$;
  \item   $\|\mathrm{div}(\bm{E}_{\mathrm{div}}^i\bm{u})\|_{H^1(\Omega)}\leq C\|\mathrm{div}\bm{u}\|_{H^1(\Omega_i)}$, with the constant $C$ depending only on $\Omega_i$;
  \item $\|\bm{E}_{\mathrm{div}}^i\bm{u}\|_{\bm{H}^s(\Omega)}\leq c \|\bm{u}\|_{\bm{H}^s(\Omega_i)},~s=1,2$, with the constant $c$ depending only on $\Omega_i$.
\end{enumerate}
\end{corollary}
 \begin{rmk}
 In the above analysis, we used the assumption that the interface $\Gamma$ is at least $C^2$-smooth to present an $\bm{H}^1(\text{div})$-extension. If the interface is nonsmooth, the trace inequality on the interface $\Gamma$ in Lemma \ref{tracedl} and Theorem \ref{divth} may not hold, and the theoretical analysis may have trouble.
\end{rmk}
Similar to Theorem \ref{h2yt}, there is an extension operator $E_i: H^s(\Omega_i)\rightarrow H^s(\Omega)$ such that $(E_iv)|_{\Omega_i}=v$ and
           \begin{equation*}
             \|E_iv\|_{s,\Omega}\leq C\|v\|_{s,\Omega_i},\quad\forall v\in H^s(\Omega_i), s=0,1.
          \end{equation*}

            Now, we turn to introducing the interpolation operator. Let $\Pi_h$ be the standard $\mathbb{BDM}_1$ interpolation operator and $I_h$ be the projection onto the space of piecewise constant functions associated with $\mathcal{T}_h$. Define
                      \begin{equation}\label{chazhi}
             (\Pi_{h}^{*}\bm{q},I_h^{*}v)=\bigg((\Pi_{h,1}^*\bm{q_1},\Pi_{h,2}^*\bm{q_2}), (I_{h,1}^*v_1,I_{h,2}^*v_2)\bigg),
           \end{equation}
           where $\Pi_{h,i}^*\bm{q_i}=(\Pi_h\bm{E}_{\mathrm{div}}^i\bm{q_i})|_{\Omega_i}$ and $I_{h,i}^*v_i=(I_hE_iv_i)|_{\Omega_i}$.

            \begin{thm}\label{czwc}
              For all $\bm{p}\in \bm{H}^1(\mathrm{div};\Omega_1\cup\Omega_2)\cap\bm{H}^2(\Omega_1\cup\Omega_2)$ and $u\in H^1(\Omega_1\cup\Omega_2)$, let $(\Pi_h^*, I_h^*)$ be the interpolation operator defined in $(\ref{chazhi})$, it is true that
              \begin{equation*}
                \|(\bm{p}-\Pi_h^*\bm{p},u-I_h^*u)\|\leq C h\bigg(\|\alpha_{\mathrm{min}}^{-\fr{1}{2}}\bm{p}\|_{2,\Omega_1\cup\Omega_2}+\|\alpha^{\fr{1}{2}}u\|_{1,\Omega_1\cup\Omega_2}
                +\|\alpha^{-\fr{1}{2}}\mathrm{div}\bm{p}\|_{1,\Omega_1\cup\Omega_2}\bigg),
                              \end{equation*}
                          where $C$ is independent of $h$ and how the interface intersects the triangulation.
                          \end{thm}
       \begin{proof}
                            Using the standard interpolation estimate of $\Pi_h$ and Corollary \ref{omegayt}, we get
                            \begin{align}\label{cz1}
                              \|\alpha^{-\fr{1}{2}}(\bm{p}-\Pi_h^*\bm{p})\|_{0,\Omega_1\cup\Omega_2}^2 \leq& \displaystyle{\sum_{i=1}^2}\displaystyle{\sum_{K\in\mathcal{T}_{h,i}}}
                              \alpha_i^{-1}\|\bm{E}_{\mathrm{div}}^i\bm{p_i}-\Pi_{h}\bm{E}_{\mathrm{div}}^i\bm{p_i}\|_{0,K}^2\notag\\
                              \leq &C\displaystyle{\sum_{i=1}^2}\displaystyle{\sum_{K\in\mathcal{T}_{h,i}}}
                              \alpha_i^{-1}h^2\|\bm{E}_{\mathrm{div}}^i\bm{p_i}\|_{1,K}^2\notag\\
                              \leq& C h^2\|\alpha^{-\frac{1}{2}}\bm{p}\|_{1,\Omega_1\cup\Omega_2}^2.
                            \end{align}

                           Combining the property (\ref{bdmjh}) of the interpolation $\Pi_h$, we have
                            \begin{align}
                              \|\mathrm{div}(\bm{p}-\Pi_h^*\bm{p})\|_{0,K_i}^2  \leq&
                              \|\mathrm{div}(\bm{E}_{\mathrm{div}}^i\bm{p_i}-\Pi_{h}\bm{E}_{\mathrm{div}}^i\bm{p_i})\|_{0,K}^2\notag
                              =\|\mathrm{div}\bm{E}_{\mathrm{div}}^i\bm{p_i}-Q_K(\mathrm{div}\bm{E}_{\mathrm{div}}^i\bm{p_i})\|_{0,K}^2\notag\\
                              \leq &C
                              h^2\|\mathrm{div}\bm{E}_{\mathrm{div}}^i\bm{p_i}\|_{1,K}^2.\notag
                            \end{align}
                            Summing over the triangles $K\in\mathcal{T}_h$ and using the property of the extension operator $\bm{E}_{\mathrm{div}}^i$ in Corollary \ref{omegayt}, we obtain
                            \begin{equation}
                             \|\alpha^{-\fr{1}{2}}\mathrm{div}(\bm{p}-\Pi_h^*\bm{p})\|_{0,\Omega_1\cup\Omega_2}^2
                             \leq
                             Ch^2\|\alpha^{-\fr{1}{2}}\mathrm{div}\bm{p}\|_{1,\Omega_1\cup\Omega_2}^2.
                            \end{equation}

                            Applying the trace inequality (\ref{traceinq}) and the standard interpolation estimate of $\Pi_h$, we derive
        \begin{align*}
        h\|[(\bm{p}-\Pi_h^*\bm{p})\cdot\bm{n}]\|_{0,\Gamma_K}^2
         \leq& Ch\sum_{i=1}^2\|(\bm{E}_{\mathrm{div}}^i\bm{p_i}-\Pi_{h}\bm{E}_{\mathrm{div}}^i\bm{p_i})\cdot{\bm{n}}\|_{0,\Gamma_K}^2\\
         \leq& Ch\sum_{i=1}^2
         \bigg(h_K^{-1}\|\bm{E}_{\mathrm{div}}^i\bm{p_i}-\Pi_{h}\bm{E}_{\mathrm{div}}^i\bm{p_i}\|_{0,K}^2\\
         &+h_K|\bm{E}_{\mathrm{div}}^i\bm{p_i}-\Pi_{h}\bm{E}_{\mathrm{div}}^i\bm{p_i}|_{1,K}^2\bigg)\notag\\
         \leq& Ch^4\sum_{i=1}^2\|\bm{E}_{\mathrm{div}}^i\bm{p_i}\|_{2,K}^2.
         \end{align*}
         Summing the contributions from $K\in G_h$, we obtain
         \begin{equation}\label{BDMP}
 h_K^{-2}\|\alpha_{\mathrm{min}}^{-\fr{1}{2}}[(\bm{p}-\Pi_h^*\bm{p})\cdot\bm{n}]\|_{-\fr{1}{2},h,\Gamma}^2
\leq Ch^2\|\alpha_{\mathrm{min}}^{-\fr{1}{2}}\bm{p}\|_{2,\Omega_1\cup\Omega_2}^2.
         \end{equation}

      Similar arguments as the above yield
          \begin{equation}
            h_K^2\|\alpha_{\mathrm{min}}^{\fr{1}{2}}\{u-I_h^*u\}\|_{\fr{1}{2},h,\Gamma}^2
            \leq Ch^2\|\alpha_{\mathrm{min}}^{\fr{1}{2}}u\|_{1,\Omega_1\cup\Omega_2}^2.
          \end{equation}

 By the standard estimate of the projection $I_h$, we deduce
         \begin{align}\label{cz2}
            \|\alpha^{\fr{1}{2}}(u-I_h^* u)\|_{0,\Omega_1\cup\Omega_2}^2
             \leq& C\displaystyle{\sum_{i=1}^2}\displaystyle{\sum_{K\in\mathcal{T}_{h,i}}}
             \alpha_i\|E_iu_i-I_hE_iu_i\|_{0,K}^2\notag\\
             \leq&C\displaystyle{\sum_{i=1}^2}\displaystyle{\sum_{K\in\mathcal{T}_{h,i}}}
             \alpha_ih^2\|E_iu_i\|_{1,K}^2\notag\\
             \leq&C h^2\|\alpha^{\fr{1}{2}}u\|_{1,\Omega_1\cup\Omega_2}^2.
          \end{align}
          Combining the inequalities (\ref{cz1})-(\ref{cz2}), the desired result follows.
           \end{proof}
\begin{rmk}
Theorem \ref{czwc} does not hold if the lowest order Raviart-Thomas finite element pair $(\mathbb{RT}_0, \mathbb{P}_0)$ is applied, since the pair $(\mathbb{RT}_0, \mathbb{P}_0)$ doesn't have second order approximation.
\end{rmk}

 Finally, we give the extended finite element error estimate as follows.

         \begin{thm}\label{yxywc}
         Assuming that $\bm{p}\in\bm{H}^1(\mathrm{div};\Omega_1\cup\Omega_2)\cap\bm{H}^2(\Omega_1\cup\Omega_2), u\in H^1(\Omega_1\cup\Omega_2)$, the solution of the problem $(\ref{femwf})$ satisfies the error estimate
         \begin{equation}
           \|(\bm{p}-\bm{p_h},u-u_h)\|\leq Ch\bigg(\|\alpha_{\mathrm{min}}^{-\fr{1}{2}}\bm{p}\|_{2,\Omega_1\cup\Omega_2}+\|\alpha^{\fr{1}{2}}u\|_{1,\Omega_1\cup\Omega_2}
            +\|\alpha^{-\fr{1}{2}}\mathrm{div}\bm{p}\|_{1,\Omega_1\cup\Omega_2}\bigg),
         \end{equation}
           where $C$ is independent of $h$ and how the interface intersects the triangulation.
       \end{thm}
\begin{proof}
       By the triangle inequality,
       \begin{align}\label{fem1}
         \|(\bm{p}-\bm{p_h},u-u_h)\|\leq & \|(\bm{p}-\Pi_h^*\bm{p},u-I_h^*u)\|+\|(\bm{p_h}-\Pi_h^*\bm{p},u_h-I_h^*u)\|_h.
       \end{align}

       Since the estimate of the first term can be obtained from Theorem $\ref{czwc}$, we only need to deal with the second term.
      \textcolor[rgb]{0.00,0.00,0.00}{Denote $\bm{\eta_h}=\bm{p_h}-\Pi_h^*\bm{p}, \zeta_h=u_h-I_h^*u$, $\bm{\eta}=\bm{p}-\Pi_h^*\bm{p}, \zeta=u-I_h^*u$. Combining Lemma $\ref{consistency}$ with Theorem $\ref{lbb}$},  we obtain
       \begin{align}\label{fem2}
          C_s\|(\bm{\eta_h},\zeta_h)\|_h\leq&\underset{(\bm{q_h},v_h)\in\bm{Q_h}\times V_h}{\mathrm{sup}}
           \dfrac{B_h(\bm{\eta_h},\zeta_h;\bm{q_h},v_h)
           +\gamma_1J_1(\zeta_h,v_h)+\gamma_2J_2(\zeta_h,v_h)}{\|(\bm{q_h},v_h)\|_h}\notag\\
           \leq&\underset{(\bm{q_h},v_h)\in\bm{Q_h}\times V_h}{\mathrm{sup}}
           \dfrac{B_h(\eta,\zeta;\bm{q_h},v_h)
           -\gamma_1J_1(I_h^*u,v_h)-\gamma_2J_2(I_h^*u,v_h)}{\|(\bm{q_h},v_h)\|_h}.
       \end{align}
      Using the Cauchy-Schwarz inequality and Lemma $\ref{fskz}$, we derive
      \begin{align}\label{fem3}
        B_h(\bm{p}-\Pi_h^*\bm{p},u-I_h^*u;\bm{q_h},v_h)\leq& C_B\|(\bm{p}-\Pi_h^*\bm{p},u-I_h^*u)\|\|(\bm{q_h},v_h)\|\notag\\
        \leq& C\|(\bm{p}-\Pi_h^*\bm{p},u-I_h^*u)\|\|(\bm{q_h},v_h)\|_h.
      \end{align}
     \textcolor[rgb]{0.00,0.00,0.00}{For the stabilization term $J_1(\cdot,\cdot)$ on (\ref{fem2}), it holds}
      \begin{align}\label{fem4}
        J_1(I_h^*u,v_h) =& \displaystyle{\sum_{i=1}^2}\displaystyle{\sum_{e\in\mathcal{F}_{h,i}^{cut}}}\int_{e_i} \alpha_{\mathrm{min}}h[I_h^*u]_{e_i}[v_h]_{e_i}ds\notag\\
        \leq&J_1(I_h^*u,I_h^*u)^{\fr{1}{2}}J_1(v_h,v_h)^{\fr{1}{2}}
        \leq J_1(I_h^*u,I_h^*u)^{\fr{1}{2}}\|(\bm{q_h},v_h)\|_h.
      \end{align}
Since $u\in H^1(\Omega_1\cup\Omega_2)$, using the trace inequality (\ref{traceinq}) and the standard estimate of $I_h$ yields
\begin{align}\label{fem5}
  J_1(I_h^*u,I_h^*u)= &J_1(u-I_h^*u,u-I_h^*u)\notag\\
  \leq& \displaystyle{\sum_{i=1}^2}\displaystyle{\sum_{e_i\in\mathcal{F}_{h,i}^{cut}}}\alpha_{\mathrm{min}}h\|E_iu_i-I_{h,i}^*u_i\|_{0,e}^2\notag\\
  \leq&C\displaystyle{\sum_{i=1}^2}\displaystyle{\sum_{K\in G_h}}\alpha_{\mathrm{min}}
  \bigg(\|E_iu_i-I_hE_iu_i\|_{0,K}^2+h^2|E_iu_i-I_hE_iu_i|_{1,K}^2\bigg)\notag\\
  \leq&C \displaystyle{\sum_{i=1}^2}\displaystyle{\sum_{K\in G_h}}\alpha_{\mathrm{min}}h^2\|E_iu_i\|_{1,K}^2\notag\\
  \leq& Ch^2\|\alpha^{\frac{1}{2}} u\|_{1,\Omega_1\cup\Omega_2}^2,
\end{align}
where $e$ is the whole edge containing the cut segment $e_i$.

Similarly,
\begin{align}\label{fem6}
  J_2(I_h^*u,v_h) =&\int_{\Gamma} \alpha_{\mathrm{min}}h[I_h^*u][v_h]ds\notag\\
  \leq&J_2(I_h^*u,I_h^*u)^{\fr{1}{2}}J_2(v_h,v_h)^{\fr{1}{2}}
  \leq J_2(I_h^*u,I_h^*u)^{\fr{1}{2}}\|(\bm{q_h},v_h)\|_h.
\end{align}
Using  the continuity of $u$ on $\Gamma$, the trace inequality (\ref{traceinq}) and the standard estimate of $I_h$, we obtain
\begin{align}\label{fem7}
  J_2(I_h^*u,I_h^*u)=& J_2(u-I_h^*u,u-I_h^*u)\notag\\
  \leq& \displaystyle{\sum_{i=1}^2}\displaystyle{\sum_{K\in G_h}}\alpha_{\mathrm{min}}h\|E_iu_i-I_{h,i}^*u_i\|_{0,\Gamma_K}^2\notag\\
  \leq&C\displaystyle{\sum_{i=1}^2}\displaystyle{\sum_{K\in G_h}}\alpha_{\mathrm{min}}\bigg(\|E_iu_i-I_hE_iu_i\|_{0,K}^2+h^2|E_iu_i-I_hE_iu_i|_{1,K}^2\bigg)\notag\\
  \leq & C\displaystyle{\sum_{i=1}^2}\displaystyle{\sum_{K\in G_h}}\alpha_{\mathrm{min}}h^2\|E_iu_i\|_{1,K}^2
  \leq  Ch^2\|\alpha^{\fr{1}{2}}u\|_{1,\Omega_1\cup\Omega_2}^2.
\end{align}
Combining the inequalities (\ref{fem1})-(\ref{fem7}), the proof is completed.
\end{proof}
\begin{rmk} We only consider the case of two subdomains in this paper. If the interfaces do not intersect with each other, it is natural to extend to multi-subdomain problems. However, it is not always so easy to extend to more than two subdomains. When it comes to multi-subdomain problems with intersecting interface, for example, problems with triple junction points, the complexity of the problems will increase.
\end{rmk}
\section{Numerical experiments}
In this section, we shall give some numerical examples to verify our theory. These numerical experiments are carried out based on the package iFEM \cite{Chenlong}. Before giving the numerical results, we define two errors $e_{u,p}$, $e_p$
\begin{eqnarray}
e_{u,p}=\fr{\|(\bm{p}-\bm{p_h},u-u_h)\|}{\|(\bm{p},u)\|_*}, ~~~
e_p=\frac{\|\bm{p}-\bm{p_h}\|_{0,\Omega}}{\|\text{div}\bm{p}\|_{1,{\Omega}}},
\end{eqnarray}
where $\|(\bm{p},u)\|_*=\|\alpha_{\mathrm{min}}^{-\fr{1}{2}}\bm{p}\|_{2,\Omega_1\cup\Omega_2}+\|\alpha^{\fr{1}{2}}u\|_{1,\Omega_1\cup\Omega_2}
            +\|\alpha^{-\fr{1}{2}}\mathrm{div}\bm{p}\|_{1,\Omega_1\cup\Omega_2}$. In the following tests, we choose the penalty parameters $\gamma=\gamma_1=\gamma_2=1$ and set the domain $\Omega=(-1,1)\times(-1,1)$ in the model problem (\ref{orig pro}).
            \textcolor[rgb]{0.00,0.00,0.00}{In Examples 1-4, we perform the numerical experiments based on structured meshes like Figure 3,  and in Example 5 numerical results are done on unstructured meshes like Figure 5. }
 \subsection{Example 1}
First we consider a circular interface $\Gamma: r_0^2=x^2+y^2$, $r_0=\frac{1}{6}$ and take the exact solution (cf. \cite{lzl03})
\begin{equation}\label{curveu}
  u=\left\{
\begin{array}{c}
\dfrac{r^5}{\alpha_1},\qquad\qquad\quad~~~~ \mathrm{if}~r\leq r_0,\vspace{3mm}\\
\dfrac{r^5}{\alpha_2}+(\dfrac{r_0^5}{\alpha_1}-\dfrac{r_0^5}{\alpha_2}),~~  \mathrm{if}~ r\geq r_0,
\end{array}
\right.
\end{equation}

\begin{equation}\label{curvep}
 \bm{p}=[5x(x^2+y^2)^{\fr{3}{2}},~5y(x^2+y^2)^{\fr{3}{2}}],
\end{equation}
 where $r^2=x^2+y^2$.
\begin{figure}[htbp]
\centering
\includegraphics[height=5cm,width=7cm]{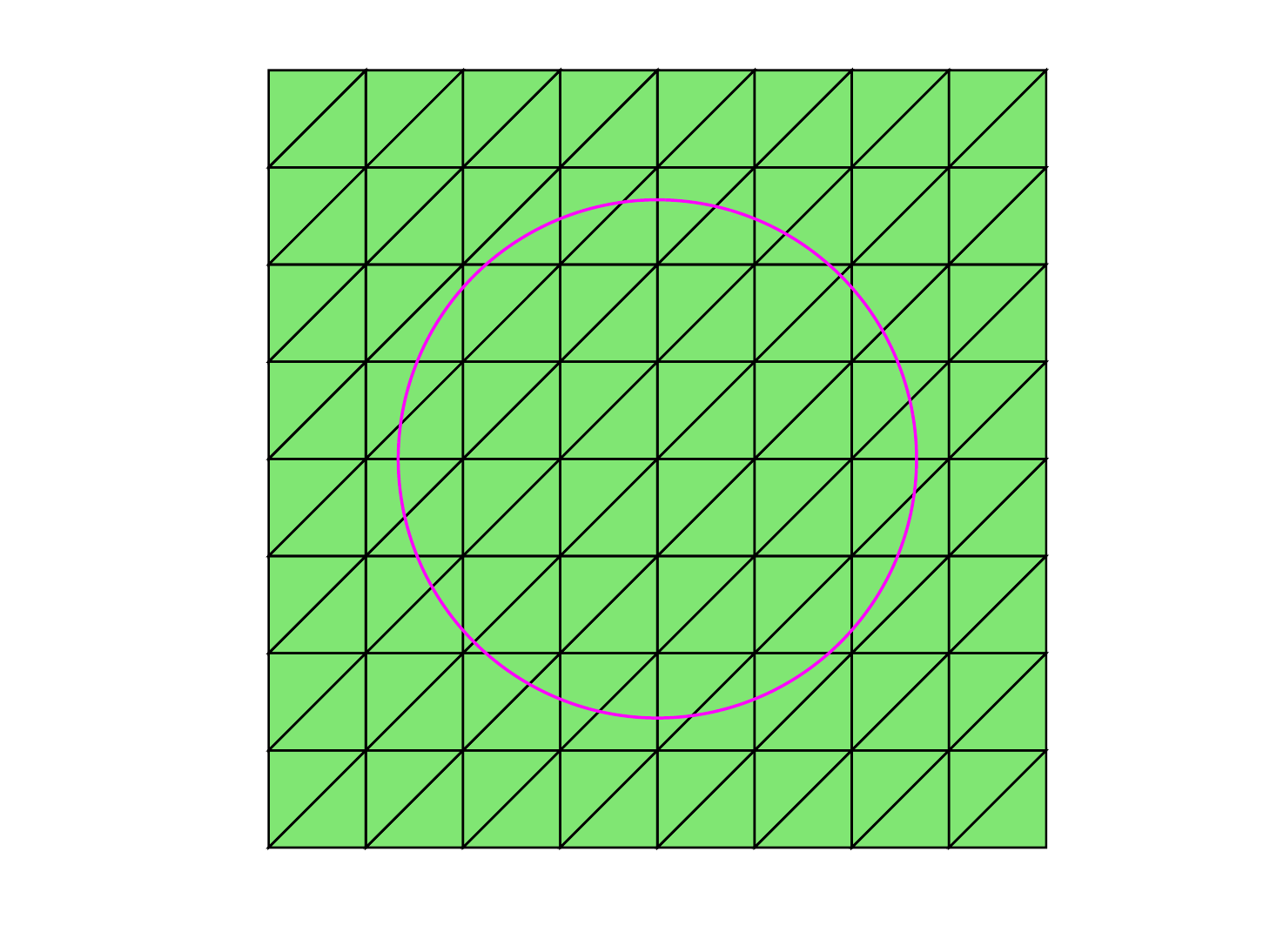}\\
{$\textbf{Fig. 3.}$ The domain with a circular interface: 8$\times$8 mesh.}
\end{figure}

\begin{table}[htbp]
\begin{center}
\caption{Errors with respect to $(\alpha_1,\alpha_2)=(1,10)$ and $(\alpha_1,\alpha_2)=(10,1)$.}
\label{curve101t}      
\begin{tabular}{lllllll}
\hline\noalign{\smallskip}
$(\alpha_1,\alpha_2)$  & ~$\frac{1}{h}$ & $e_p$ &order &$e_{u,p}$ &order\\
\noalign{\smallskip}\hline\noalign{\smallskip}
         &~~8&1.0605e-03 &          &1.8310e-02   &      \\
       & ~16 &2.7170e-04  &1.9647 &9.2224e-03&  0.9895    \\
 (1,10)& ~32 &6.8494e-05 &1.9879 &4.6197e-03 &  0.9974\\
       & ~64 &1.7178e-05& 1.9954 &2.3109e-03 &   0.9993\\
        &128 &  4.3000e-06& 1.9981 &1.1556e-03 & 0.9998\\
 \hline
             &~~8&3.3532e-03& & 4.1614e-02&  \\
               &~16&8.5914e-04 &  1.9646  &   2.0961e-02&  0.9894\\
       (10,1) &~32& 2.1659e-04 &  1.9880 &  1.0499e-02 &   0.9974\\
              &~64& 5.4317e-05 &   1.9955 &5.2521e-03 &  0.9994\\
               &128& 1.3597e-05&   1.9981 &   2.6263e-03&    0.9998\\
\noalign{\smallskip}\hline
\end{tabular}
\end{center}
\end{table}
\begin{table}[htbp]
\begin{center}
\caption{Errors with respect to $(\alpha_1,\alpha_2)=(1,10^3)$ and $(\alpha_1,\alpha_2)=(10^3,1)$.}
\label{curve1031t}       
\begin{tabular}{lllllll}
\hline\noalign{\smallskip}
$(\alpha_1,\alpha_2)$  & ~$\frac{1}{h}$ & $e_p$ &order &$e_{u,p}$ &order\\
\noalign{\smallskip}\hline\noalign{\smallskip}
             &~~8&1.0623e-04  &      & 7.0681e-04 &\\
           & ~16 &  2.7198e-05&  1.9657 &   3.5607e-04&  0.9892\\
 (1,$10^3$)& ~32 &  6.8639e-06 &   1.9864 &  1.7856e-04&    0.9958  \\
           & ~64 &  1.7246e-06 &  1.9928 &    8.9409e-05&  0.9979 \\
             &128 &  4.3210e-07&  1.9968  &   4.4729e-05&  0.9992\\
 \hline
                 &~~8&3.3532e-03&  &4.1614e-02&\\
                 &~16&   8.5914e-04& 1.9646 &  2.0971e-02&  0.9887 \\
       ($10^3$,1)&~32&    2.1658e-04&   1.9880 &   1.0504e-02&    0.9975  \\
                 &~64&  5.4316e-05&   1.9955  &  5.2528e-03&    0.9997  \\
                &128&    1.3597e-05&  1.9981 &    2.6264e-03&   0.9999 \\
\noalign{\smallskip}\hline
\end{tabular}
\end{center}
\end{table}
\begin{table}[htbp]
\begin{center}
\caption{Errors with respect to $(\alpha_1,\alpha_2)=(1,10^5)$ and $(\alpha_1,\alpha_2)=(10^5,1)$.}
\label{curve1051t}      
\begin{tabular}{lllllll}
\hline\noalign{\smallskip}
$(\alpha_1,\alpha_2)$  & ~$\frac{1}{h}$ & $e_p$ &order &$e_{u,p}$ &order\\
\noalign{\smallskip}\hline\noalign{\smallskip}
           &~~8&     1.0810e-05&          & 7.9172e-06&         \\
           & ~16 &   2.9890e-06&1.8547 &4.2835e-06  & 0.8862     \\
(1,$10^5$) & ~32 &   8.1830e-07&1.8690 &2.3229e-06 &   0.8829 \\
           & ~64 &   2.3066e-07&1.8269 &  1.2389e-06&  0.9069\\
           &128 &    6.0589e-08&1.9286 & 6.3539e-07&   0.9634  \\
 \hline
              &~~8&3.3532e-03&               & 4.1614e-02\\
              &~16&    8.5914e-04 &1.9646 &  2.1985e-02 &  0.9206\\
 ($10^5$,1)   &~32&   2.1658e-04   &   1.9880 &  1.0917e-02&    1.0099 \\
              &~64&   5.4316e-05 &   1.9955 &   5.3244e-03 &  1.0359  \\
              &128&   1.3597e-05&   1.9981 &  2.6366e-03&  1.0139 \\
\noalign{\smallskip}\hline
\end{tabular}
\end{center}
\end{table}

\textcolor[rgb]{0.00,0.00,0.00}{In
Tables \ref{curve101t}-\ref{curve1051t}, we list the errors with respect to different mesh sizes $\big($$h=\dfrac{1}{N},~N=$ 8, 16, 32, 64, 128$\big)$ and different pairs of coefficients $(\alpha_1,\alpha_2)=$(1,10), (10,1), (1,$10^3$),
 ($10^3$,1), (1,$10^5$), ($10^5$,1).} From 
Tables \ref{curve101t}-\ref{curve1051t}, we can observe that the convergence of $\bm{p}-\bm{p_h}$ in the norm $\|\cdot\|_h$ and $u-u_h$ in the norm $\|\cdot\|_{\star}$ are optimal. Moreover, the convergence of $\bm{p}-\bm{p_h}$ in the norm $\|\cdot\|_0$ is second order.

\subsection{Example 2}
  Next we consider an elliptic interface $\frac{x^2}{a^2}+\frac{y^2}{b^2}=1$, where $a=\frac{1}{6},~b=\frac{1}{7}$. The exact solution reads as
  \begin{equation}\label{ellipu}
    u=\left\{
\begin{array}{c}
\dfrac{(\fr{x^2}{a^2}+\fr{y^2}{b^2})^{\fr{5}{2}}}{\alpha_1},\qquad\qquad\quad~~~~ \mathrm{if}~ \frac{x^2}{a^2}+\frac{y^2}{b^2}\leq 1,\vspace{3mm}\\
\dfrac{(\fr{x^2}{a^2}+\fr{y^2}{b^2})^{\fr{5}{2}}}{\alpha_2}+(\dfrac{1}{\alpha_1}-\dfrac{1}{\alpha_2}),~~  \mathrm{if}~ \frac{x^2}{a^2}+\frac{y^2}{b^2}\geq 1,
\end{array}
\right.
  \end{equation}
\begin{equation}\label{ellipp}
  \bm{p}=[\fr{5x}{a^2}(\fr{x^2}{a^2}+\fr{y^2}{b^2})^{\fr{3}{2}},~
     \fr{5y}{b^2}(\fr{x^2}{a^2}+\fr{y^2}{b^2})^{\fr{3}{2}}].
\end{equation}


%
%
\begin{table}[htbp]
\begin{center}
\caption{Errors with respect to $(\alpha_1,\alpha_2)=(1,10)$ and $(\alpha_1,\alpha_2)=(10,1)$.}
\label{ellip101t}      
\begin{tabular}{lllllll}
\hline\noalign{\smallskip}
$(\alpha_1,\alpha_2)$  & ~$\frac{1}{h}$ & $e_p$ &order &$e_{u,p}$ &order\\
\noalign{\smallskip}\hline\noalign{\smallskip}
         &~~8& 1.0459e-03&        &1.8201e-02    &      \\
       & ~16 & 2.6787e-04&1.9651 & 9.1705e-03&0.9889 \\
(1,10) & ~32 & 6.7518e-05&1.9882 &4.5940e-03& 0.9972 \\
       & ~64 & 1.6931e-05&1.9956 &2.2981e-03&0.9993\\
        &128 & 4.2382e-06&1.9982 &1.1492e-03&  0.9998  \\
 \hline
 &~~8&3.3070e-03& &4.1497e-02&\\
       &~16& 8.4703e-04&  1.9650 &2.0909e-02&0.9889\\
(10,1)       &~32& 2.1350e-04& 1.9882  &1.0474e-02&0.9973\\
       &~64& 5.3538e-05&1.9956 & 5.2396e-03&0.9993\\
       &128&1.3402e-05&1.9982 &2.6201e-03&0.9998\\
\noalign{\smallskip}\hline
\end{tabular}
\end{center}
\end{table}
\begin{table}[htbp]
\begin{center}
\caption{Errors with respect to $(\alpha_1,\alpha_2)=(1,10^3)$ and $(\alpha_1,\alpha_2)=(10^3,1)$.}
\label{ellip1031t}       
\begin{tabular}{lllllll}
\hline\noalign{\smallskip}
$(\alpha_1,\alpha_2)$  & ~$\frac{1}{h}$ & $e_p$ &order &$e_{u,p}$ &order\\
\noalign{\smallskip}\hline\noalign{\smallskip}
           &~~8&1.0477e-04&  & 7.1725e-04 &\\
           & ~16 & 2.6806e-05& 1.9666 &3.6137e-04&0.9890\\
(1,$10^3$) & ~32 & 6.7610e-06 &1.9872 &1.8112e-04&0.9965\\
           & ~64 &1.6968e-06 &1.9944 & 9.0650e-05&0.9986\\
           &128 & 4.2504e-07 &1.9971 & 4.5347e-05&0.9993 \\
 \hline
           &~~8&3.3069e-03& &4.1497e-02&\\
           &~16&8.4702e-04&1.9650 & 2.0919e-02&0.9882 \\
 ($10^3$,1)&~32&2.1350e-04&1.9882 & 1.0478e-02&  0.9974 \\
            &~64& 5.3538e-05&1.9956 &5.2402e-03& 0.9997\\
           &128&  1.3402e-05&1.9982 & 2.6202e-03& 0.9999 \\
\noalign{\smallskip}\hline
\end{tabular}
\end{center}
\end{table}

\begin{table}[htbp]
\begin{center}
\caption{Errors with respect to $(\alpha_1,\alpha_2)=(1,10^5)$ and $(\alpha_1,\alpha_2)=(10^5,1)$.}
\label{ellip1051t}      
\begin{tabular}{lllllll}
\hline\noalign{\smallskip}
$(\alpha_1,\alpha_2)$  & ~$\frac{1}{h}$ & $e_p$ &order &$e_{u,p}$ &order\\
\noalign{\smallskip}\hline\noalign{\smallskip}
        &~~8&1.0556e-05 &         &7.8668e-06&\\
        & ~16 &2.8586e-06 &1.8846  &4.1953e-06&0.9070\\
(1,$10^5$) & ~32 &  7.6259&1.9063  &2.1895e-06&0.9382\\
        & ~64 & 2.0283e-07&1.9106 &1.1399e-06& 0.9417\\
        &128 &  5.3241e-08&1.9297 &5.8480e-07&  0.9629 \\
 \hline
          &~~8&3.3069e-03&          & 4.1497e-02 &  \\
          &~16&8.4702e-04&1.9650 &2.1882e-02& 0.9233 \\
($10^5$,1)&~32&2.1350e-04&1.9882 &1.0871e-02&  1.0093 \\
          &~64& 5.3538e-05&1.9956 & 5.2995e-03& 1.0365  \\
          &128& 1.3402e-05&1.9982 &2.6284e-03&1.0117 \\
\noalign{\smallskip}\hline
\end{tabular}
\end{center}
\end{table}
 Numerical results are shown in
  Tables \ref{ellip101t}-\ref{ellip1051t}, we can see the same convergence property holds as Example 1.

\subsection{Example 3}
\textcolor[rgb]{0.00,0.00,0.00}{ In order to test the robustness with respect to the position of the interface}, we consider a family of interfaces $\Gamma:~x+\frac{\pi}{6}+\xi=0$, and $\xi$ varies from 0.1 to 0.0000001. The exact solution is
  \begin{equation}\label{ellipu}
    u=\left\{
\begin{array}{c}
\dfrac{(x+\frac{\pi}{6}+\xi)(x^2-1)(y^2-1)}{\alpha_1}, \quad\mathrm{if} ~x\leq -\frac{\pi}{6}-\xi,\vspace{3mm}\\
\dfrac{(x+\frac{\pi}{6}+\xi)(x^2-1)(y^2-1)}{\alpha_2}, \quad\mathrm{if} ~x\geq -\frac{\pi}{6}-\xi,
\end{array}
\right.
  \end{equation}
  \quad\quad$\bm{p}=[p_1,p_2],$
  where
  \begin{align*}
    p_1= & (\frac{\pi}{6}+\xi + x)(x - 1)(y^2 - 1) + (\frac{\pi}{6}+\xi + x)(x + 1)(y^2 - 1)\\
    & + (x^2 - 1)(y^2 - 1),\\
    p_2= & (\frac{\pi}{6}+\xi + x)(x^2 - 1)(y - 1) + (\frac{\pi}{6}+\xi + x)(x^2 - 1)(y + 1).
  \end{align*}
\begin{figure*}[htbp]
  \center{
  \includegraphics[width=8cm]{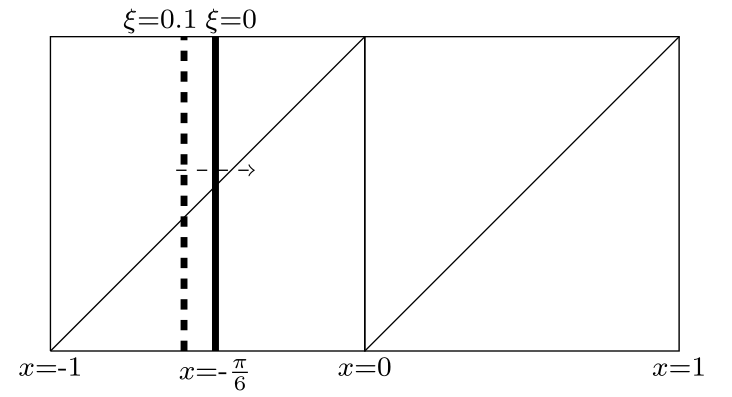}}
  \centerline{$\textbf{Fig. 4.}$  The position of the interface.}
\end{figure*}
\begin{table}[htbp]
\begin{center}
\caption{Relative errors with a family of interfaces.}
\label{moving interface}       
\begin{tabular}{lllll}
\hline\noalign{\smallskip}
$(\alpha_1,\alpha_2)$  & ~$\xi$  &$~~~~~e_{u,p}$&$~~~~~~e_p$ \\
\noalign{\smallskip}\hline\noalign{\smallskip}
 &0.1 & 1.1497e-03&8.3966e-05 \\
 & 0.01&1.3532e-03&9.5288e-05 \\
(1,$10^3$) &0.001&1.3795e-03 & 9.6553e-05 \\
 & 0.0001 &1.3825e-03&  9.6723e-05\\
 &0.00001 &1.3828e-03&  9.6740e-05 \\
 &0.000001& 1.3828e-03&  9.6742e-05\\
 &0.0000001&1.3828e-03&9.6742e-05\\
\noalign{\smallskip}\hline
\end{tabular}
\end{center}
\end{table}

Taking the pair $(\alpha_1,\alpha_2)$ as (1,$10^3$) and the mesh size $h=\fr{1}{32}$, Table \ref{moving interface} shows the relative errors of the solutions are below a certain threshold when $\xi$ varies, which indicates the errors are almost independent of the interface to some extent.

\subsection{Example 4}
In this test, we consider the case where the exact solution is discontinuous on the interface $\Gamma: x=\fr{1}{3}$. The exact solution is
  \begin{equation}\label{ellipu}
    u=\left\{
\begin{array}{c}
y^3+2, \quad\mathrm{if} ~x\leq \frac{1}{3},\vspace{3mm}\\
2y^3, \quad\quad\mathrm{if} ~ x> \frac{1}{3},
\end{array}
\right.
  \end{equation}
  and
  \begin{equation}\label{ellipu}
    \bm{p}=\left\{
\begin{array}{lc}
\left[0,3y^2\right], \quad\mathrm{if} ~x\leq \frac{1}{3},\vspace{3mm}\\
\left[0,6y\right], \quad~\mathrm{if} ~ x> \frac{1}{3}.
\end{array}
\right.
  \end{equation}

  \begin{table}[htbp]
  \begin{center}
\caption{Errors with respect to $(\alpha_1,\alpha_2)=(1,10)$ and $(\alpha_1,\alpha_2)=(10,1)$.}
\label{jumpline101t}      
\begin{tabular}{lllllll}
\hline\noalign{\smallskip}
$(\alpha_1,\alpha_2)$  & ~$\frac{1}{h}$ & $~~~~~~e_p$ &order &$~~~~~e_{u,p}$
&order\\
\noalign{\smallskip}\hline\noalign{\smallskip}
           &~~8 &8.0233e-04&       &1.5624e-02                 \\
           & ~16&1.8001e-04&2.1561&7.8922e-03&0.9853    \\
 (1,$10$)  & ~32&4.4847e-05&2.0050&3.8741e-03&1.0265   \\
           & ~64&1.1052e-05&2.0207&1.9471e-03&0.9926    \\
           &128 &2.7611e-06&2.0010&9.7105e-04&1.0037     \\
 \hline
                &~~8&7.3742e-04&       &2.8865e-04                          \\
                &~16&1.7974e-04&2.0365 &1.4069e-04&1.0368                        \\
       ($10$,1) &~32&4.3585e-05&2.0441 &7.0319e-05&1.0005    \\
                &~64&1.0828e-05&2.0090 &3.5058e-05&1.0042   \\
                &128&2.6993e-06&2.0041 &1.7541e-05&0.9990      \\
\noalign{\smallskip}\hline
\end{tabular}
\end{center}
\end{table}
    \begin{table}[htbp]
    \begin{center}
\caption{Errors with respect to $(\alpha_1,\alpha_2)=(1,10^3)$ and $(\alpha_1,\alpha_2)=(10^3,1)$.}
\label{jumpline1031t}      
\begin{tabular}{lllllll}
\hline\noalign{\smallskip}
$(\alpha_1,\alpha_2)$  & ~$\frac{1}{h}$ & $~~~~~~e_p$ &order &$~~~~~e_{u,p}$
&order\\
\noalign{\smallskip}\hline\noalign{\smallskip}
           &~~8 &2.5516e-05&       &5.0312e-04                \\
           & ~16&5.6891e-06&2.1651&2.5562e-04&0.9769      \\
 (1,$10^3$)& ~32&1.4282e-06&1.9940&1.2541e-04&1.0273   \\
           & ~64&3.5170e-07&2.0218&6.3103e-05&0.9909    \\
           &128 &8.8023e-08&1.9984&3.1467e-05&1.0038     \\
 \hline
                &~~8&2.1700e-05&       &8.4026e-07                          \\
                &~16&5.4647e-06&1.9895 &4.0896e-07&1.0389                      \\
      ($10^3$,1)&~32&1.3275e-06&2.0415 &2.0559e-07&0.9922    \\
                &~64&3.3245e-07&1.9974 &1.0244e-07&1.0049\\
                &128&8.2923e-08&2.0033 &5.1309e-08&0.9975     \\
\noalign{\smallskip}\hline
\end{tabular}
\end{center}
\end{table}
  \begin{table}[htbp]
  \begin{center}
\caption{Errors with respect to $(\alpha_1,\alpha_2)=(1,10^5)$ and $(\alpha_1,\alpha_2)=(10^5,1)$.}
\label{jumpline1051t}      
\begin{tabular}{lllllll}
\hline\noalign{\smallskip}
$(\alpha_1,\alpha_2)$  & ~$\frac{1}{h}$ & $~~~~~~e_p$ &order &$~~~~~e_{u,p}$
&order\\
\noalign{\smallskip}\hline\noalign{\smallskip}
           &~~8 &2.5517e-06&      &5.0312e-05&                                 \\
           & ~16&5.6891e-07&2.1652&2.5562e-05& 0.9769  \\
 (1,$10^5$)& ~32&1.4282e-07&1.9939&1.2541e-05&1.0273    \\
           & ~64&3.5170e-08&2.0218&6.3104e-06&0.9909   \\
           &128 &8.8024e-09&1.9984&3.1468e-06&1.0038        \\
 \hline
                &~~8&2.1698e-06&      &8.4017e-10                            \\
                &~16&5.4645e-07&1.9894&4.0892e-10 &1.0389   \\
      ($10^5$,1)&~32&1.3274e-07&2.0415&2.0557e-10&0.9922   \\
                &~64&3.3245e-08&1.9974&1.0243e-10 &1.0049  \\
                &128&8.2921e-09&2.0033& 5.1304e-11&0.9975                  \\
\noalign{\smallskip}\hline
\end{tabular}
\end{center}
\end{table}
From tables \ref{jumpline101t}-\ref{jumpline1051t}, we can see that the same convergence property holds as Example 1 even though the exact solution is discontinuous on the interface, which is consistent with our theoretical analysis.
\begin{figure}[htbp]
\centering
\includegraphics[height=5cm,width=7cm]{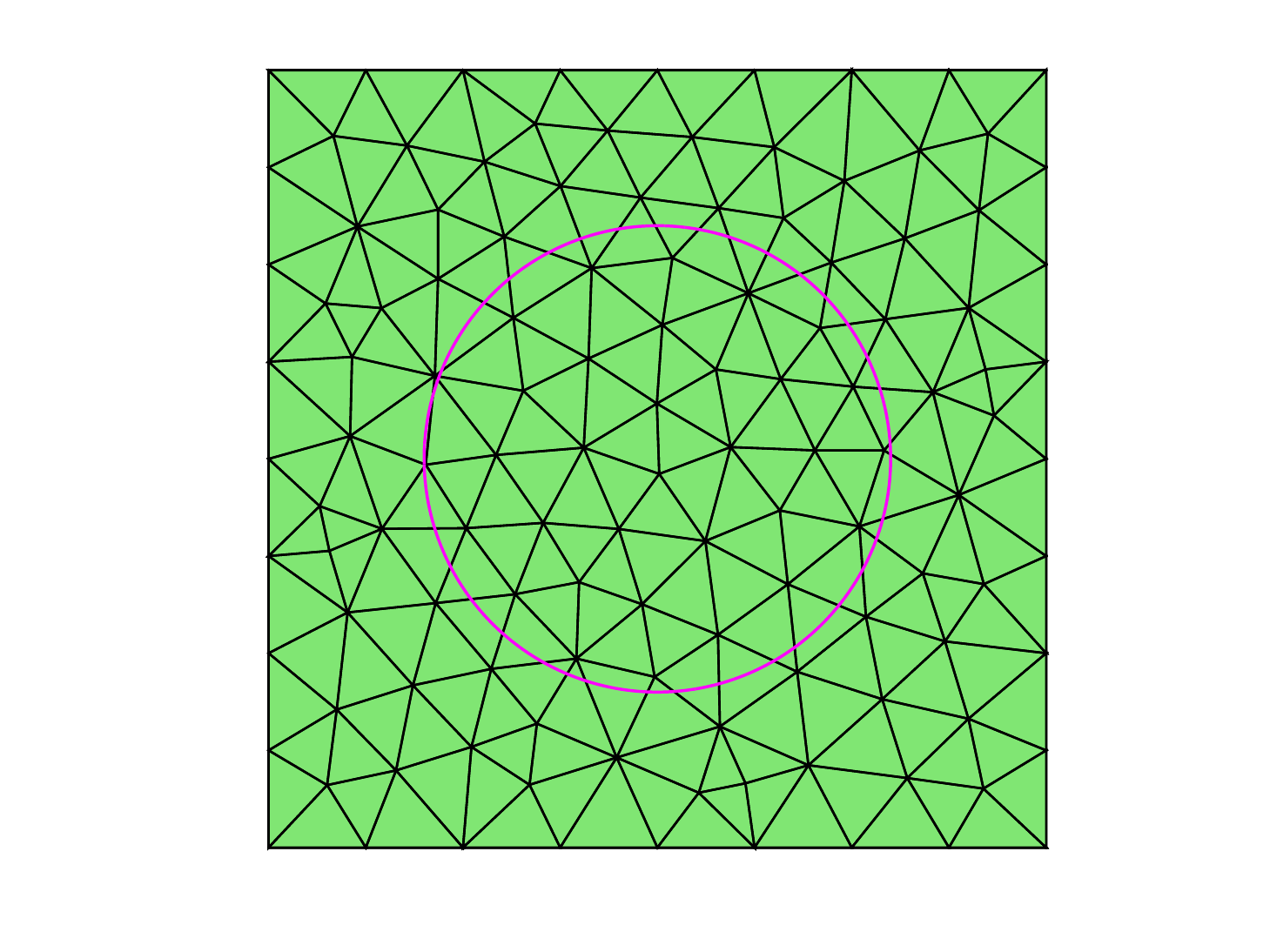}\\
{$\textbf{Fig. 5.}$ The domain with a circular interface: 8$\times$8 mesh.}
\end{figure}
\subsection{Example 5}
In this test, we consider the case where the exact solution is discontinuous on the interface $\Gamma: x=\fr{1}{3}$. The exact solution is
  \begin{equation}\label{ellipu}
    u=\left\{
\begin{array}{c}
\dfrac{(x^2-1)(y^2-1)}{\alpha_1}, \quad\mathrm{if} ~x\leq \frac{1}{3},\vspace{3mm}\\
\dfrac{(x^2-1)(y^2-1)}{\alpha_2}, \quad\mathrm{if} ~x> \frac{1}{3},
\end{array}
\right.
  \end{equation}
  $\bm{p}=[p_1,p_2],$
  where
  \begin{align*}
    p_1= & (x - 1)(y^2 - 1) + (x + 1)(y^2 - 1),\\
    p_2= & (x^2 - 1)(y - 1) + (x^2 - 1)(y + 1),
  \end{align*}
  and
  \begin{equation}
    [u]=\bigg(\frac{1}{\alpha_1}-\frac{1}{\alpha_2}\bigg)(x^2-1)(y^2-1)\neq 0,\quad \mathrm{if} \quad\alpha_1\neq\alpha_2.
  \end{equation}

 \begin{table}[htbp]
 \begin{center}
\caption{Errors with respect to $(\alpha_1,\alpha_2)=(1,10)$ and $(\alpha_1,\alpha_2)=(10,1)$.}
\label{hgline101t}      
\begin{tabular}{lllllll}
\hline\noalign{\smallskip}
$(\alpha_1,\alpha_2)$  & ~$\frac{1}{h}$ & $~~~~~~e_p$ &order &$~~~~~e_{u,p}$
&order\\
\noalign{\smallskip}\hline\noalign{\smallskip}
           &~~8&2.2900e-03   &             &2.1246e-02                         \\
           & ~16 &5.2361e-04 & 2.1287    &  1.0654e-02 &0.9959                           \\
 (1,$10$)  & ~32 & 1.2917e-04  & 2.0192  &4.6577e-03 & 1.1936   \\
           & ~64 &3.1558e-05 & 2.0332 & 2.2500e-03&  1.0497   \\
           &128 &7.8296e-06&  2.0110   & 1.1026e-03& 1.0290    \\
 \hline
                &~~8&1.7887e-03&          &1.8474e-02               \\
                &~16&4.4033e-04&2.0222 &9.0719e-03&1.0260       \\
       ($10$,1) &~32&1.0252e-04&2.1027 &4.1444e-03&1.1303       \\
                &~64&2.5361e-05&2.0152 &1.9490e-03&1.0884   \\
                &128&6.2378e-06&2.0235 &9.5173e-04&1.0341     \\
\noalign{\smallskip}\hline
\end{tabular}
\end{center}
\end{table}
 \begin{table}[htbp]
 \begin{center}
\caption{Errors with respect to $(\alpha_1,\alpha_2)=(1,10^3)$ and $(\alpha_1,\alpha_2)=(10^3,1)$.}
\label{hgline1031t}      
\begin{tabular}{lllllll}
\hline\noalign{\smallskip}
$(\alpha_1,\alpha_2)$  & ~$\frac{1}{h}$ & $~~~~~~e_p$ &order &$~~~~~e_{u,p}$
&order\\
\noalign{\smallskip}\hline\noalign{\smallskip}
           &~~8&2.0681e-03   &            &6.7820e-03                            \\
           & ~16 &4.8083e-04  &2.1047  &3.3307e-03& 1.0259                                   \\
 (1,$10^3$)  & ~32 &1.2020e-04&2.0001  &1.4754e-03& 1.1747  \\
           & ~64 &2.9688e-05 & 2.0175  &7.1944e-04& 1.0362  \\
           &128 &7.3936e-06 &  2.0055  &3.5338e-04& 1.0256 \\
 \hline
                &~~8&1.3589e-03&            &4.2089e-03               \\
                &~16&3.5136e-04  &1.9514 &2.0056e-03& 1.0694                         \\
       ($10^3$,1) &~32&8.6122e-05&2.0285 &9.3737e-04&1.0973        \\
                &~64& 2.1568e-05 &1.9975 &4.4292e-04&1.0816    \\
                &128&  5.3779e-06&2.0038 &2.1721e-04&1.0280  \\
\noalign{\smallskip}\hline
\end{tabular}
\end{center}
\end{table}
 \begin{table}[htbp]
 \begin{center}
\caption{Errors with respect to $(\alpha_1,\alpha_2)=(1,10^5)$ and $(\alpha_1,\alpha_2)=(10^5,1)$.}
\label{hgline1051t}      
\begin{tabular}{lllllll}
\hline\noalign{\smallskip}
$(\alpha_1,\alpha_2)$  & ~$\frac{1}{h}$ & $~~~~~~e_p$ &order &$~~~~~e_{u,p}$
&order\\
\noalign{\smallskip}\hline\noalign{\smallskip}
             &~~8  &2.0656e-03  &         &7.1412e-04    \\
             & ~16 &4.8035e-04 &2.1044 &3.5064e-04&1.0262    \\
 (1,$10^5$)  & ~32 &1.2009e-04 &2.0000 &1.5534e-04&1.1745    \\
             & ~64 &2.9667e-05 &2.0171 &7.5756e-05&1.0360    \\
             &128  &7.3889e-06 &2.0054 &3.7211e-05&1.0256    \\
 \hline
                  &~~8&1.3533e-03&           &4.3109e-04              \\
                  &~16&3.5019e-04&1.9503 &2.0536e-04&1.0699     \\
       ($10^5$,1) &~32&8.5917e-05&2.0271 &9.6015e-05&1.0968        \\
                  &~64&2.1523e-05&1.9970 &4.5374e-05&1.0814    \\
                  &128&5.3690e-06&2.0032 &2.2253e-05&1.0279    \\
\noalign{\smallskip}\hline
\end{tabular}
\end{center}
\end{table}
From tables \ref{hgline101t}-\ref{hgline1051t}, we can see that the same convergence property holds as Example 1 even though the numerical experiment is carried out on unstructured meshes, which is in compliance with our theoretical analysis.

\section{Conclusions}
In this paper, we propose an extended mixed finite element method for elliptic interface problems with interface-unfitted meshes. An inf-sup stability result and optimal error estimates have been derived which are independent of the mesh size and how the interface intersects the triangulation. The optimal convergence is shown by theories and numerical experiments. Moreover, 2-order convergence rate is verified by numerical results for $\bm{p}-\bm{p_h}$ in the norm $\|\cdot\|_0$.

In the future, we shall study the variant coefficient case, the three-dimensional case and also
apply our proposed scheme to $\bm{H}(\text{div})$-like interface problems, such as Darcy-Stokes interface problems, which has important applications in the simulation of flow in porous media.
\section*{Acknowledgements}
This work was supported by the National Natural Science Foundation of China (NSFC) (Grants Nos. 11871281, 11731007, 12071227), and the Natural Science Foundation of the Jiangsu Higher Education Institutions of China (Grants Nos.  20KJA110001). The authors thank the reviewers for their valuable suggestions to improve this manuscript greatly.


\clearpage
\clearpage
\clearpage
{\small
\bibliographystyle{unsrt}

}
\end{document}